\documentclass[11pt,a4paper,dvipsnames]{article}
\usepackage[top=35mm,bottom=32mm,right=20mm,left=20mm]{geometry}
\usepackage[T1]{fontenc}
\usepackage[utf8]{inputenc}
\usepackage{mathptmx,lmodern,mathrsfs,amsthm,amsmath,graphicx,graphics, float,cases,xfrac,amssymb,bm,bbm,mathtools,gensymb,tikz,tocloft,listings,pgffor,lipsum,microtype,titling,mathrsfs,proof,sectsty,fancyhdr,longtable,lscape,emptypage,parskip,booktabs,stmaryrd,enumerate,autobreak}
\usepackage[labelfont={bf,sf}]{caption}
\usepackage{subcaption}
\usepackage[%
maxnames=20,
giveninits=true,
isbn=false,
]{biblatex}
\usepackage{pgfplots}
\usepgfplotslibrary{colormaps}
\usepackage{pgfplotstable}
\pgfplotsset{compat=1.7}
\usepgfplotslibrary{fillbetween}
\definecolor{matlabcolor1}{rgb}{0,0.4470,0.7410}
\definecolor{matlabcolor2}{rgb}{0.8500,0.3250,0.0980}
\definecolor{matlabcolor3}{rgb}{0.9290,0.6940,0.1250}
\definecolor{matlabcolor4}{rgb}{0.4940,0.1840,0.5560}
\definecolor{matlabcolor5}{rgb}{0.4660,0.6740,0.1880}
\definecolor{matlabcolor6}{rgb}{0.3010, 0.7450, 0.9330}
\definecolor{matlabcolor7}{rgb}{0.6350, 0.0780, 0.1840}

\definecolor{matlabcolor1op}{rgb}{0.5,0.7235,0.8705}
\definecolor{matlabcolor2op}{rgb}{0.8500,0.3250,0.0980}
\definecolor{matlabcolor3op}{rgb}{0.9290,0.6940,0.1250}
\definecolor{matlabcolor4op}{rgb}{0.497,0.4537,0.7132}
\definecolor{matlabcolor5op}{rgb}{0.4660,0.6740,0.1880}
\definecolor{matlabcolor6op}{rgb}{0.3010, 0.7450, 0.9330}
\definecolor{matlabcolor7op}{rgb}{0.4955, 0.3188, 0.6346}

\pgfplotsset{
/pgfplots/colormap={mycolormap}{rgb=(0.4495,0.2869,0.5711) rgb=(0.8994,0.8907,0.9426)}
}

\let \mathbf \bm
\DeclareMathAlphabet\mathbfcal{OMS}{cmsy}{b}{n} 

\newcommand\norm[1]{\left\lVert#1\right\rVert}
\newcommand\inner[2]{\left\langle #1, #2 \right\rangle}
\newcommand{\reals}{\mathbb{R}}

\newcommand{\naturals}{\mathbb{N}_0}

\newcommand{\sym}{\mathbb{S}}

\newcommand{\calH}{\mathcal{H}}

\newcommand{\trace}{{\mathrm{trace}}}

\newcommand\diag{\operatorname{diag}}
\newcommand\rank{\operatorname{rank}}

\newcommand{\prox}{{\rm{prox}}}
\newcommand{\kron}{\otimes}

\DeclareMathOperator*{\argmin}{argmin}

\DeclareMathOperator*{\zer}{zer}
\DeclareMathOperator*{\dom}{dom}
\DeclareMathOperator*{\Id}{Id}

\newcommand{\quadform}[2]{\mathcal{Q}(#1,#2)}
\newcommand{\XId}[1]{#1_{\Id}}

\newcommand{\xmiddle}[1]{\;\middle#1\;}

\newcommand{\bx}{\mathbf{x}}
\newcommand{\bu}{\mathbf{u}}
\newcommand{\by}{\mathbf{y}}
\newcommand{\bz}{\mathbf{z}}
\newcommand{\bfcn}{\mathbf{f}}
\newcommand{\bFcn}{\mathbf{F}}
\newcommand{\bM}{\mathbf{M}}
\newcommand{\bMlij}{\bM_{(l,i,j)}}
\newcommand{\ba}{\mathbf{a}}
\newcommand{\balij}{\mathbf{a}_{(l,i,j)}}
\newcommand{\bzeta}{\boldsymbol{\zeta}}
\newcommand{\bchi}{\boldsymbol{\chi}}
\newcommand{\bxi}{\boldsymbol{\xi}}

\newcommand{\bQ}{\mathbf{Q}}
\newcommand{\bq}{\mathbf{q}}

\newcommand{\SumToZeroMat}{N}

\newcommand{\indentconstr}{\;\;\;}

\newcommand\Update[1]{#1}  
\newcommand\set[1]{\mathord{\left\{ #1 \right\}}}
\newcommand\p[1]{\mathord{\left( #1 \right)}}

\allsectionsfont{\normalfont\sffamily\bfseries}

\makeatletter
\def\th@plain{
  \thm@headfont{\normalfont\sffamily\bfseries}
  \itshape 
}

\def\th@definition{
  \thm@headfont{\normalfont\sffamily\bfseries}
  \thm@notefont{\normalfont\sffamily\bfseries}
}

\newtheoremstyle{myStyle1}
  {0.3cm}
  {0.3cm}
  {\itshape}
  {}
  {\normalfont\sffamily\bfseries}
  {:}
  {.5em}
  {\thmname{#1}\ \thmnumber{#2}\thmnote{ \textnormal{(#3)}}}
\makeatother

\newtheoremstyle{myStyle2}
  {0.3cm}
  {0.3cm}
  {}
  {}
  {\normalfont\sffamily\bfseries}
  {:}
  {.5em}
  {\thmname{#1}\ \thmnumber{#2}\thmnote{ \textnormal{(#3)}}}

\theoremstyle{myStyle1}
\newtheorem{theorem}{Theorem}
\newtheorem{proposition}{Proposition}
\newtheorem{corollary}{Corollary}
\newtheorem{assumption}{Assumption}
\newtheorem{lemma}{Lemma}
\newtheorem{definition}{Definition}

\theoremstyle{myStyle2}
\newtheorem{remark}{Remark}

\providecommand{\keywords}[1]
{\textbf{\textsf{Keywords.}} #1}

\usepackage[hidelinks]{hyperref}

\title{\Large \sffamily\bfseries Automated tight Lyapunov analysis for first-order methods}

\author{Manu Upadhyaya$^{\star}$ \and Sebastian Banert$^{\star}$ \and Adrien B. Taylor$^{\dagger}$ \and Pontus Giselsson$^{\star}$}

\date{%
    $^{\star}$Department of Automatic Control\\%
    Lund University, Lund, Sweden\\%
    \{\href{mailto:manu.upadhyaya@control.lth.se}{manu.upadhyaya}, \href{mailto:sebastian.banert@control.lth.se}{sebastian.banert}, \href{mailto:pontus.giselsson@control.lth.se}{pontus.giselsson}\}@control.lth.se\\[2ex]%
    $^{\dagger}$INRIA \& D.I. \'Ecole Normale Supérieure\\
    CNRS \& PSL Research University, Paris, France\\
    \href{mailto:adrien.taylor@inria.fr}{adrien.taylor@inria.fr}
}

\begin{document}
\maketitle

\begin{abstract}
    \noindent 
We present a methodology for establishing the existence of quadratic Lyapunov inequalities for a wide range of first-order methods used to solve convex optimization problems. In particular, we consider i) classes of optimization problems of finite-sum form with (possibly strongly) convex and possibly smooth functional components, ii) first-order methods that can be written as a linear system in state-space form in feedback interconnection with the subdifferentials of the functional components of the objective function, and iii) quadratic Lyapunov inequalities that can be used to draw convergence conclusions. \Update{We present a necessary and sufficient condition for the existence of a quadratic Lyapunov inequality within a predefined class of Lyapunov inequalities, which amounts to solving a small-sized semidefinite program.} We showcase our methodology on several first-order methods that fit the framework. Most notably, our methodology allows us to significantly extend the region of parameter choices that allow for duality-gap convergence in the Chambolle--Pock method when the linear operator is the identity mapping.  
\end{abstract}


\section{Introduction}
First-order methods are used to solve optimization problems and can be analyzed via Lyapunov inequalities. Such inequalities consist of a Lyapunov function that is nonincreasing from one iteration to the next and a residual function that quantifies a lower bound on the potential decrease. The traditional approach of establishing a Lyapunov inequality, which is typically done on a case-by-case basis, amounts to combining and rearranging algorithm update equations and inequalities that describe properties of the objective function. In this paper, we develop an automated methodology for finding Lyapunov inequalities that can be applied to a large class of first-order methods.

The methodology uses an algorithm representation that covers most first-order methods with fixed parameters. The structure of the algorithm representation is a linear system in state-space form in feedback interconnection with a nonlinearity, in our case, the subdifferentials of the functional components of the objective function. Such representations are common in the automatic control literature \cite{Essentials1998Zhou} and have previously been used for algorithm analysis, e.g., in \cite{Lessard_2016}. The algorithm representation is also closely connected to the operator splitting framework introduced in \cite{Morin2022Frugal}. Different algorithms are obtained by instantiating the matrices that define the linear system. Some matrix choices lead to algorithms that cannot solve the optimization problem in general. A contribution of this paper is that we provide conditions on the matrices that are necessary and sufficient for the equivalence between solving an instance of the optimization problem and finding a fixed point of the algorithm.

Our methodology is based on a necessary and sufficient condition for the existence of a quadratic Lyapunov inequality \Update{within a predefined class of Lyapunov inequalities}. 
At the core of the methodology is a necessary and sufficient condition, in terms of a semidefinite program, for the optimal value of a quadratic objective function to be nonpositive when optimized over all possible algorithm iterates, fixed points, subgradients, and function values and over the full function class under consideration. This result is applied to the three conditions that we use to define a quadratic Lyapunov inequality. The resulting semidefinite program is feasible if and only if such a quadratic Lyapunov inequality exists, and it provides associated Lyapunov functions and residual functions when feasible.

Other methodologies that analyze optimization algorithms using semidefinite programs are the performance estimation problem (PEP) methodology~\cite{Drori_2014,Taylor_Hendrickx_Glineur_2017a} and the integral quadratic constraints (IQC) methodology~\cite{Lessard_2016}. The PEP methodology poses the problem of finding a worst-case function from a predefined class of functions for the algorithm under consideration as an optimization problem. This is then reformulated in a sequence of steps to arrive at a semidefinite program. The PEP methodology, first presented in~\cite{Drori_2014}, has been extended in a sequence of works that guarantee tightness in each step of the reformulation~\cite{Taylor_Hendrickx_Glineur_2017a,Taylor_Hendrickx_Glineur_2017b}, has been adapted as a tool for Lyapunov analysis~\cite{TaylorBach_2019,TaylorLessard_2018,Celine_cont_time}, and extended to monotone inclusion problems~\cite{Ryu_2020}. \Update{The IQC methodology is based on integral quadratic constraints from the control literature \cite{Megretski1997System}, which has been adopted for automated convergence analysis of first-order methods under various settings \cite{Lessard_2016,vanscoy2021speedrobustness,Lessard_Dissipativity_2022}. The IQC methodology uses a simple algorithm representation but lacks tightness guarantees.} We are inspired by the strengths of both methodologies; the worst-case analysis and tightness guarantees of PEP and the simple algorithm representation of IQC. Another work that is inspired by the PEP and IQC frameworks for tight Lyapunov function analysis is~\cite{TaylorLessard_2018}. Our framework is more general as it can be applied to a wider range of algorithms, allowing, e.g., for proximal operators, and can be used to derive a broader range of convergence results.

The proposed methodology is applied in two ways. First, to find the smallest possible linear convergence rate \Update{via quadratic Lyapunov inequalities} for the algorithm at hand. This is done via a bisection search over the convergence rate $\rho\in[0,1[$. Second, to find the range of algorithm parameters for which the Lyapunov analysis can guarantee function-value convergence or duality-gap convergence. The algorithms we consider are the Douglas--Rachford method~\cite{Douglas_Rachford_numerical_1956,Lions_Mercier_SplittingAlgorithmsSum_1979}, the (proximal) gradient method with heavy-ball momentum~\cite{Polyak1987Introduction,Ghadimi2015Global}, the three-operator splitting method by Davis and Yin \cite{Davis2017AThreeOperator}, and the Chambolle--Pock method \cite{Chambolle_Pock_ergodic_2011}.

For the Douglas--Rachford method, we recover some of the known tight linear convergence rate results in \cite{Giselsson2017Linear,Giselsson2017Tight,Ryu_2020}. For the gradient method with heavy-ball momentum, we improve, compared to \cite{Ghadimi2015Global}, the linear convergence rate, and also extend the range of parameters that guarantee convergence in function-value suboptimality. We also show convergence of the duality gap for two proximal gradient methods with heavy-ball momentum. For the three-operator splitting method by Davis and Yin, we provide linear convergence rate results that improve the ones found in \cite{Davis_Yin_3op_arxiv,pmlr-v80-pedregosa18a}. More strikingly, our methodology allows us to significantly enlarge the range of parameters that give duality-gap convergence for the Chambolle--Pock method when the linear operator is assumed to be the identity operator. Traditional proofs, such as in \cite{Chambolle_Pock_ergodic_2011}, allow for proximal operator step-size parameters $\tau_1,\tau_2>0$ to satisfy $\tau_1\tau_2<1$ and the coefficient $\theta$ for the linear combination of previous iterates to satisfy $\theta=1$. Our analysis allows for a significantly wider range of parameter values, e.g., for $\theta=1$ we allow for $\tau_1=\tau_2\in]0,1.15]$, for $\theta=0.35$ we allow for $\tau_1=\tau_2\in]0,1.5]$, and for $\tau_1=\tau_2=0.5$, we allow for $\theta\in[0.03,7.5]$. We also demonstrate, through our methodology, that the extended range of parameters can lead to improved linear convergence rates over the traditional parameter choices.

The paper is organized as follows: in Section~\ref{sec:problem_and_algorithm}, we introduce the problem class and the algorithm representation. Section~\ref{sec:interpolation} discusses interpolation results and frames them in our setting. We define the notion of a quadratic Lyapunov inequality in Section~\ref{sec:Lyapunov}. Section~\ref{sec:main} contains the main result on the existence of a quadratic Lyapunov inequality. Section~\ref{sec:numerical_examples} contains numerical examples and Section~\ref{sec:lemma_1_proof} contains a proof of our core result. Section~\ref{sec:conclusions} contains the main conclusions of this work and discusses future work.

An implementation of the methodology and additional numerical examples can be found at:
\begin{center}
    \url{https://github.com/ManuUpadhyaya/TightLyapunovAnalysis}
\end{center}

\subsection{Preliminaries}
\label{sec:prelim}

Let $\naturals$ denote the set of nonnegative integers, $\mathbb{Z}$ the set of integers, $\llbracket n,m \rrbracket = \set{l \in \mathbb{Z} \xmiddle\vert n \leq l \leq m}$ the set of integers inclusively between the integers $n$ and $m$,  $\reals$ the set of real numbers, $\reals_+$ the set of nonnegative real numbers, $\reals_{++}$ the set of positive real numbers, $\reals^n$ the set of all $n$-tuples of elements of $\reals$, $\reals^{m\times n}$ the set of real-valued matrices of size $m\times n$, if $M\in\reals^{m\times n}$ then $[M]_{i,j}$ the $i,j$-th element of $M$, $\sym^{n}$ the set of symmetric real-valued matrices of size $n\times n$, and $\sym_+^n\subseteq \sym^{n}$ the set of positive semidefinite real-valued matrices of size $n\times n$. $\mathbf{1}$ denotes the column vector of all ones, where the size will be clear from the context.

Throughout this paper, $(\calH,\inner{\cdot}{\cdot})$ will denote a real Hilbert space. All norms $\norm{\cdot}$ are canonical norms where the inner product will be clear from the context. We denote the identity mapping $x \mapsto x$ on $\calH$ by $\Id$. Given a function $f:\calH\rightarrow\reals\cup\{+\infty\}$, the \emph{effective domain} of $f$ is the set $\dom f = \left\{x\in\calH \xmiddle\vert f(x)<+\infty\right\}$. The function $f$ is said to be \emph{proper} if $\dom f \neq \emptyset$. The \emph{subdifferential} of a proper function $f$ is the set-valued operator $\partial f:\calH \rightarrow 2^{\calH}$ defined as the mapping $x \mapsto \set{u \in \calH \xmiddle\vert \forall y \in \calH,\, f(y) \geq f(x) + \inner{u}{y-x} }$.

Let $f\colon \calH \to \reals \cup \{+\infty\}$ and $\sigma,\beta\in\reals_{+}$. The function $f$ is 
\begin{itemize}
    \item[(i)] \emph{convex} if $f\p{\p{1 - \lambda} x + \lambda y} \leq \p{1 - \lambda} f\p{x} + \lambda f\p{y}$ for each $x, y \in \calH$ and $0 \leq \lambda \leq 1$,
    \item[(ii)] \emph{$\sigma$-strongly convex} if $f$ is proper and $f-(\sigma/2)\norm{\cdot}^{2}$ is convex, and
    \item[(iii)] \emph{$\beta$-smooth} if $f$ is differentiable and $\norm{\nabla f(x) - \nabla f(y)} \leq \beta \norm{x - y}$ for each $x,y\in\calH$.
\end{itemize}

Let $0 \leq \sigma < \beta \leq +\infty$. We let $\mathcal{F}_{\sigma,\beta}$ denote the class of all functions $f:\calH\rightarrow\reals\cup\{+\infty\}$ that are
\begin{itemize}
    \item[(i)] $\beta$-smooth and $\sigma$-strongly convex if $\beta < +\infty$, and
    \item[(ii)] lower semicontinuous and $\sigma$-strongly convex if $\beta = +\infty$.
\end{itemize}

Let $f:\calH\to \reals \cup \{+\infty\}$ be proper, lower semicontinuous and convex, and let $\gamma>0$. Then the \emph{proximal operator} $\prox_{\gamma f} : \calH \to \calH $ is defined as the single-valued operator given by
\begin{align*}
    \prox_{\gamma f}(x) = \argmin_{z\in\calH}\p{f(z) + \frac{1}{2\gamma}\norm{x-z}^2} 
\end{align*}
for each $x\in\calH$. If $x,p\in\calH$, then $p = \prox_{\gamma f}(x)$ $\Leftrightarrow$ $\gamma^{-1}\p{x-p} \in \partial f (p)$. Moreover, the \emph{conjugate} of $f$, denoted $f^{*}:\calH\to \reals \cup \{+\infty\}$, is the proper, lower semicontinuous and convex function given by $f^{*}(u) = \sup_{x\in\calH}\p{\inner{u}{x} - f(x)}$ for each $u\in\calH$. If $x,u\in\calH$, then $u \in \partial f(x) $ $\Leftrightarrow$ $x \in \partial f^{*}(u)$ \cite[Theorem 16.29]{Bauschke_Combettes_2017}.

Given any positive integer $n$, we let the inner product $\inner{\cdot}{\cdot}$ on $\calH^{n}$ be given by $\inner{\bz_1}{\bz_2}=\sum_{j=1}^{n}\inner{z_1^{(j)}}{z_2^{(j)}}$ for each $\bz_i=\p{z_{i}^{(1)},\ldots,z_{i}^{(n)}}\in\calH^{n}$ and $i\in\llbracket 1,2\rrbracket$. If $M\in \reals^{m\times n}$, we define the tensor product $M\kron \Id$ to be the mapping $(M\kron \Id):\calH^{n}\rightarrow\calH^{m}$ such that 
\begin{align*}
  (M\kron\Id)\bz = \left(\sum_{j=1}^n[M]_{1,j}z^{(j)},\ldots,\sum_{j=1}^n[M]_{m,j}z^{(j)}\right)
\end{align*}
for each $\bz=\p{z^{(1)},\ldots,z^{(n)}}\in\calH^n$. The adjoint satisfies $(M\kron\Id)^*=M^{\top}\kron\Id$. If $N\in\reals^{n\times l}$, the composition rule $(M\kron\Id)\circ(N\kron\Id)=(MN)\kron\Id$ holds. Moreover, if $M\in\reals^{n\times n}$ is invertible, then $(M\kron \Id)^{-1} = M^{-1}\kron \Id$ holds.

If we let $M_{1}\in\reals^{m\times n_{1}}$ and $M_{2}\in\reals^{m\times n_{2}}$, the relations above imply that $\inner{(M_1\kron\Id)\bz_{1}}{(M_2\kron\Id)\bz_{2}}=\inner{\bz_{1}}{\p{\p{M_1^{\top}M_2}\kron\Id}\bz_{2}}$ for each $\bz_1\in\calH^{n_1}$ and $\bz_2\in\calH^{n_2}$. We define the mapping\footnote{We use the same symbol $\mathcal{Q}$ for the mapping independent of the dimension $n$, which will be clear from context.}
$\mathcal{Q}:\sym^{n}\times \calH^{n} \rightarrow \reals$ by $\quadform{M}{\bz}=\inner{\bz}{(M\kron\Id)\bz}$ for each $M\in\sym^{n}$ and $\bz\in\calH^{n}$. Note that, if $M\in\sym^{n}$, $N\in\reals^{n\times m}$ and $\bz\in\calH^{m}$, then $\quadform{M}{(N\kron\Id)\bz}=\quadform{N^{\top}MN}{\bz}$.  

\section{Problem class and algorithm representation}\label{sec:problem_and_algorithm}
In this section, we introduce the problem class and the algorithm representation. We provide conditions for when solving a problem is equivalent to finding a fixed point of an algorithm. We also provide conditions for when an algorithm can be implemented using scalar multiplications, vector additions, proximal operator evaluations, and gradient evaluations only. We conclude the section by listing a few examples of first-order methods that fit into the algorithm representation. 

\subsection{Problem class}

Let $0\leq \sigma_{i} < \beta_{i} \leq +\infty$ for each $i\in\llbracket1,m\rrbracket$. Consider the convex optimization problem
\begin{equation}
\begin{aligned}\label{eq:the_problem}
& \underset{y\in\calH}{\text{minimize}} & & \sum_{i=1}^{m}f_{i}(y)
\end{aligned}
\end{equation}
where $f_i\in\mathcal{F}_{\sigma_i,\beta_i}$ for each $i\in\llbracket1,m\rrbracket$. Most first-order methods are limited to solving the related inclusion problem
\begin{align}
\label{eq:the_problem_inclusion}
\text{find}\  y\in\calH\ \text{ such that }\ 0\in\sum_{i=1}^m\partial f_i(y).
\end{align}
A solution to~\eqref{eq:the_problem_inclusion} is always a solution to~\eqref{eq:the_problem} and the converse holds under some appropriate constraint qualification, e.g., see~\cite{Bot_2010}. Moreover, it is reasonable to only consider problems such that the inclusion problem~\eqref{eq:the_problem_inclusion} is solvable, i.e., there exists at least one point $y\in\calH$ such that $0\in\sum_{i=1}^m\partial f_i(y)$. Thus, the problem class we consider is all solvable problems of the form \eqref{eq:the_problem_inclusion} where $f_i\in\mathcal{F}_{\sigma_i,\beta_i}$ for each $i\in\llbracket1,m\rrbracket$. \Update{For examples of problems that can be modeled according to~\eqref{eq:the_problem} or~\eqref{eq:the_problem_inclusion}, we refer to the textbooks \cite{Beck_book_2017,Nesterov_LecturesConvexOptimization_2018}.}

For later convenience, we introduce the notation 
\begin{align*}
    \zer \p{ \sum_{i=1}^m\partial f_i } = \set{y \in \calH \xmiddle\vert 0\in\sum_{i=1}^m\partial f_i(y)}.
\end{align*}
That is, $\zer \p{ \sum_{i=1}^m\partial f_i }$ is the set of zeros of the set-valued operator $\sum_{i=1}^m\partial f_i:\calH \rightarrow 2^{\calH}:y\mapsto \sum_{i=1}^m\partial f_i(y)$, which is the same as the set of solutions to~\eqref{eq:the_problem_inclusion}.

\subsection{Algorithm representation}\label{subsec:state_space}

We consider algorithms that solve \eqref{eq:the_problem_inclusion} that can be represented as a discrete-time linear system in state-space form in feedback interconnection with the potentially nonlinear and set-valued subdifferentials that define the problem. In particular, let $\bfcn:\calH^m\to(\reals\cup\{+\infty\})^m$ and $\bm{\partial}\bfcn:\calH^m\to 2^{\calH^m}$ be mappings containing all functions and subdifferentials associated with \eqref{eq:the_problem} and \eqref{eq:the_problem_inclusion} that satisfy
\begin{align}
    \label{eq:product_f}\bfcn(\by)&=\p{f_1\p{y^{(1)}},\ldots,f_{m}\p{y^{(m)}}},\\
    \label{eq:product_f_subdiff}\bm{\partial}\bfcn(\by)&=\prod_{i=1}^{m}\partial f_i\p{y^{(i)}}
\end{align}
for each $\by=\p{y^{(1)},\ldots,y^{(m)}}\in\calH^{m}$, respectively\footnote{In this context, the symbol $\prod$ is used for Cartesian products.}. We consider algorithms that can be written as: pick an initial $\bx_{0}\in\calH^{n}$ and let
\begin{equation}\label{eq:linear_system_with_nonlinearity}
    \begin{aligned}
        &\text{for } k = 0,1,\ldots \\
        &\left\lfloor
        \begin{aligned}
            &\bx_{k+1} = \p{A \kron \Id} \bx_{k} + \p{B \kron \Id} \bu_{k}, \\
            &\by_{k} =  \p{C \kron \Id} \bx_{k} + \p{D \kron \Id} \bu_{k}, \\
            &\bu_{k} \in \bm{\partial f}(\by_{k}),\\
            &\bFcn_{k} = \bfcn(\by_{k}),
        \end{aligned} \right.
    \end{aligned}
\end{equation}
where $\bx_{k}\in\calH^n$, $\bu_{k}\in\calH^m$, $\by_{k}\in\calH^m$, and $\bFcn_{k}\in\reals^m$ are the algorithm variables and
\begin{align*}
A&\in\reals^{n\times n},& B&\in\reals^{n\times m},& C&\in\reals^{m\times n},& D&\in\reals^{m\times m}
\end{align*}
are fixed matrices containing the parameters of the method at hand. For clarification, individual subgradients and function values are calculated as $u_k^{(i)}\in\partial f_i\p{y_k^{(i)}}$ and $\bFcn_k^{(i)}=f_i\p{y_k^{(i)}}$, respectively, so that $\bu_k=\p{u_k^{(1)},\ldots,u_k^{(m)}}$ and $\bFcn_k=\p{f_1\p{y_k^{(1)}},\ldots,f_m\p{y_k^{(m)}}}$. Moreover, \Update{representation~\eqref{eq:linear_system_with_nonlinearity} is a tool for analysis and does not necessarily indicate an efficient implementation, e.g.,} the function values are not used in the algorithm but are needed for the Lyapunov analysis. The structure in \eqref{eq:linear_system_with_nonlinearity} of a linear system in feedback interconnection with a nonlinearity is common in the automatic control literature and has previously been proposed in \cite{Lessard_2016,Zhao_Lessard_Udell_2021} as a model for algorithm analysis. It can represent a wide range of first-order methods as seen in Section~\ref{sec:alg_examples}.

Algorithm~\eqref{eq:linear_system_with_nonlinearity} searches for a \emph{fixed point} $(\bx_\star,\bu_\star,\by_\star,\bFcn_\star)\in\calH^n\times\calH^m\times\calH^m\times\reals^m$ satisfying the fixed-point equations
\begin{equation}\label{eq:fixed_point}
    \begin{aligned} 
        \bx_{\star} &= \p{A \kron \Id} \bx_{\star} + \p{B \kron \Id} \bu_{\star}, \\
        \by_{\star} &=  \p{C \kron \Id} \bx_{\star} + \p{D \kron \Id} \bu_{\star}, \\
        \bu_{\star} &\in \bm{\partial f}(\by_{\star}),\\
        \bFcn_{\star} &= \bfcn(\by_{\star}),
    \end{aligned} 
\end{equation}
from which we want to recover a solution to~\eqref{eq:the_problem_inclusion}. In particular, we want the problem of finding a fixed point of \eqref{eq:linear_system_with_nonlinearity} to be equivalent to solving \eqref{eq:the_problem_inclusion}.

\subsection{Solutions and fixed points}\label{sec:fixed-point_encodings}

There are choices of the matrices $A$, $B$, $C$, and $D$ such that it is not possible to extract a solution of~\eqref{eq:the_problem_inclusion} from fixed points of~\eqref{eq:linear_system_with_nonlinearity} in any practical way. To exclude such algorithms, we add the requirement that fixed points should satisfy
\begin{align}
\by_{\star}&=(y_\star,\ldots,y_\star), &0&=\sum_{i=1}^m u_{\star}^{(i)} 
\label{eq:sol_from_fp}
\end{align}
for some $y_\star\in\calH$, where $\bu_{\star}=\p{u_{\star}^{(1)},\ldots,u_{\star}^{(m)}}\in\calH^{m}$. This implies that $y_\star$ solves \eqref{eq:the_problem_inclusion} since the fixed-point equations \eqref{eq:fixed_point} give that
\begin{align*}
    0=\sum_{i=1}^m u_{\star}^{(i)}\in\sum_{i=1}^m\partial f_i(y_\star).
\end{align*}
We say that such fixed points are {\emph{fixed-point encodings}} in line with the terminology in \cite{Ryu2020Uniqueness}. By defining the set of fixed points as
\begin{align*}
    \Omega_{\textup{fixed points}}\p{f_1,\ldots,f_m} = \set{ (\bx_\star,\bu_\star,\by_\star,\bFcn_\star)\in\calH^n\times\calH^m\times\calH^m\times\reals^m \xmiddle\vert \eqref{eq:fixed_point}\text{ holds}}
\end{align*}
and the set fixed-point encodings as
\begin{align*}
    \Omega_{\textup{fixed-point encodings}}\p{f_1,\ldots,f_m} = \set{ (\bx_\star,\bu_\star,\by_\star,\bFcn_\star)\in\calH^n\times\calH^m\times\calH^m\times\reals^m \xmiddle\vert \eqref{eq:fixed_point}\text{ and }\eqref{eq:sol_from_fp}\text{ hold}},
\end{align*}
the requirement that all fixed points are fixed-point encodings can be written as
$\Omega_{\textup{fixed points}}\p{f_1,\ldots,f_m}=\Omega_{\textup{fixed-point encodings}}\p{f_1,\ldots,f_m}$. Another requirement is that to each solution 
 of \eqref{eq:the_problem_inclusion}, there exists a corresponding fixed point. These two requirements imply that solving \eqref{eq:the_problem_inclusion} is equivalent to finding a fixed point of the algorithm. We say that such algorithms have the {\emph{fixed-point encoding property}}.

\begin{definition}[\normalfont Fixed-point encoding property]\label{def:FP_E_property}
    We say that algorithm~\eqref{eq:linear_system_with_nonlinearity} has the \emph{fixed-point encoding property} if 
    \begin{align}\label{eq:in_def_all_sols_have_fp_encodings}
        y_{\star} \in \zer \p{ \sum_{i=1}^m\partial f_i } \implies \exists \p{\bx_{\star},\bu_{\star},\bFcn_{\star}}\in\calH^n\times\calH^m\times\reals^m \text{ such that } \eqref{eq:fixed_point}\text{ and }\eqref{eq:sol_from_fp}\text{ hold},
    \end{align}
    and
    \begin{align}\label{eq:in_def_all_fps_are_fp_encodings}
        \Omega_{\textup{fixed points}}\p{f_1,\ldots,f_m} = \Omega_{\textup{fixed-point encodings}}\p{f_1,\ldots,f_m}
    \end{align}
    for each $\p{f_{1},\ldots,f_m}\in\prod_{i=1}^{m}\mathcal{F}_{\sigma_{i},\beta_{i}}$.
\end{definition}

By appropriately restricting $A$, $B$, $C$, and $D$, we can exactly capture the class of algorithms with this property. For $m\geq 2$, let 
 \begin{align}\label{eq:SumToZeroMat_def}
    \SumToZeroMat = 
        \begin{bmatrix}
            I\\
            -\mathbf{1}^{\top}
        \end{bmatrix}\in\reals^{m\times (m-1)}
        \quad{\text{ and }}\quad  
        \hat{\bu}_\star&=\p{u_{\star}^{(1)},\ldots,u_{\star}^{(m-1)}}.
    \end{align}
The fixed-point encoding condition in \eqref{eq:sol_from_fp} is then equivalent to $0=\p{\SumToZeroMat^{\top}\kron\Id}\by_\star$ and $\bu_{\star}=\p{\SumToZeroMat\kron\Id}\hat{\bu}_\star$. In the case $m=1$, the fixed-point encoding condition is simply $\bu_{\star}=0$. The matrix $\SumToZeroMat$ enters in the restriction of $A$, $B$, $C$, and $D$ to exactly capture the class of algorithms with the fixed-point encoding property.

\begin{assumption}\label{ass:ABCD_solution_and_fixed_point_condition}
    Suppose that 
    \begin{align}\label{eq:ABCD_range_condition}
        \mathrm{ran}
        \begin{bmatrix}
            B\SumToZeroMat & 0\\
            D\SumToZeroMat & -\mathbf{1}   
        \end{bmatrix}
        \subseteq
        \mathrm{ran}
        \begin{bmatrix}
            I-A\\
            -C   
        \end{bmatrix}
    \end{align}
    with the interpretation that the block column containing $\SumToZeroMat$ is removed when $m=1$, and that
    \begin{align}\label{eq:ABCD_null_condition}
        \mathrm{null}
        \begin{bmatrix}
            I-A \;\;&\;\;  -B
        \end{bmatrix}
        \subseteq
        \mathrm{null}
        \begin{bmatrix}
            \SumToZeroMat^{\top} C & \SumToZeroMat^{\top} D \\
            0 & \mathbf{1}^{\top}
        \end{bmatrix},
    \end{align}
    with the interpretation that the block row containing $\SumToZeroMat^{\top}$ is removed when $m=1$. 
\end{assumption}

\begin{proposition}\label{prp:solution_and_fixed_point}
    The following are equivalent:
    \begin{itemize}
        \item[(i)] Assumption~\ref{ass:ABCD_solution_and_fixed_point_condition} holds.
        \item[(ii)] Algorithm~\eqref{eq:linear_system_with_nonlinearity} has the fixed-point encoding property.
    \end{itemize}
\end{proposition}

\begin{proof}
    \textit{(i) $\Rightarrow$ (ii):}    
        Suppose that Assumption~\ref{ass:ABCD_solution_and_fixed_point_condition} holds.
        Let $\p{f_{1},\ldots,f_m}\in\prod_{i=1}^{m}\mathcal{F}_{\sigma_{i},\beta_{i}}$.

        First, we prove that~\eqref{eq:in_def_all_sols_have_fp_encodings} holds.
        Suppose that $y_{\star} \in \zer \p{ \sum_{i=1}^m\partial f_i }$. This implies that there exists a $\bu_{\star}=\p{u_{\star}^{(1)},\ldots,u_{\star}^{(m)}}\in\calH^{m}$ such that 
        \begin{align}\label{eq:in_proof_solution}
            \bu_{\star} &\in \bm{\partial f}\p{\p{\mathbf{1} \kron \Id} y_{\star}} \quad\text{ and }\quad \sum_{i=1}^{m}u_{\star}^{(i)} = 0.
        \end{align}
        Note that the second part of~\eqref{eq:in_proof_solution} implies that 
        \begin{align*}
            \bu_{\star} =
            \begin{cases}
                \p{\SumToZeroMat\kron\Id} \hat{\bu}_{\star} & \text{if } m>1, \\
                0 & \text{if } m=1,
            \end{cases}
        \end{align*}
        for $\hat{\bu}_{\star}=\p{u_{\star}^{(1)},\ldots,u_{\star}^{(m-1)}}\in\calH^{m-1}$, where $\SumToZeroMat$ is defined in~\eqref{eq:SumToZeroMat_def}.
        We will show that there exists an $\bx_{\star}\in\calH^{n}$ such that 
        \begin{equation}\label{eq:fixed_point_from_solution_1}
            \begin{aligned} 
                \bx_{\star} &= \p{A \kron \Id} \bx_{\star} + \p{B \kron \Id} \bu_{\star}, \\
                \p{\mathbf{1} \kron \Id} y_{\star} &=  \p{C \kron \Id} \bx_{\star} + \p{D \kron \Id} \bu_{\star},
            \end{aligned} 
        \end{equation}
        i.e., 
        \begin{align}\label{eq:fixed_point_from_solution_2}
            \p{\bx_{\star}, \bu_{\star}, \p{\mathbf{1} \kron \Id} y_{\star}, \bfcn\p{\p{\mathbf{1} \kron \Id} y_{\star}}}
        \end{align}
        is a fixed-point encoding. 
        This will prove the desired implication.
        Note that \eqref{eq:fixed_point_from_solution_1} is equivalent to  
        \begin{align*}
            \p{
            \begin{bmatrix}
                B\SumToZeroMat & 0 \\
                D\SumToZeroMat & -\mathbf{1} 
            \end{bmatrix} \kron \Id
            } \p{\hat{\bu}_{\star}, y_{\star}}
            =
            \p{
            \begin{bmatrix}
                I-A \\
                -C 
            \end{bmatrix} \kron \Id
            } \bx_{\star},
        \end{align*}
        with the interpretation that $\hat{\bu}_{\star}$ and the block column containing $\SumToZeroMat$ is removed when $m=1$.
        Moreover,~\eqref{eq:ABCD_range_condition} in Assumption~\ref{ass:ABCD_solution_and_fixed_point_condition} implies that there exists a matrix $U\in\reals^{n\times m}$ such that 
        \begin{align*}
            \begin{bmatrix}
                B\SumToZeroMat & 0\\
                D\SumToZeroMat & -\mathbf{1}   
            \end{bmatrix}
            = \begin{bmatrix}
            I-A\\
            -C   
        \end{bmatrix} U,
        \end{align*}
        i.e., each column of the matrix to the left in~\eqref{eq:ABCD_range_condition} can be written as linear combinations of the columns of the matrix to the right in~\eqref{eq:ABCD_range_condition}. 
        If we let $\bx_{\star}=\p{U\kron\Id}\p{\hat{\bu}_{\star}, y_{\star}}$, then we get 
        \begin{align*}
            \p{
            \begin{bmatrix}
                B\SumToZeroMat & 0 \\
                D\SumToZeroMat & -\mathbf{1} 
            \end{bmatrix} \kron \Id
            } \p{\hat{\bu}_{\star}, y_{\star}}
            &=
            \p{\p{
            \begin{bmatrix}
                I-A \\
                -C 
            \end{bmatrix}U} \kron \Id
            } \p{\hat{\bu}_{\star}, y_{\star}} \\
            &=
            \p{
            \begin{bmatrix}
                I-A \\
                -C 
            \end{bmatrix} \kron \Id
            } \p{U \kron \Id } \p{\hat{\bu}_{\star}, y_{\star}} \\
            &=
            \p{
            \begin{bmatrix}
                I-A \\
                -C 
            \end{bmatrix} \kron \Id
            } \bx_{\star},
        \end{align*}
        as desired.
        
        Second, we prove that~\eqref{eq:in_def_all_fps_are_fp_encodings} holds.
        Note that the inclusion $\supseteq$ holds trivially. Therefore, we only need to prove the $\subseteq$ inclusion. Suppose that $(\bx_\star,\bu_\star,\by_\star,\bFcn_\star)\in\Omega_{\textup{fixed points}}\p{f_1,\ldots,f_m}$. This implies that 
        \begin{align*}
            \p{
            \begin{bmatrix}
                I-A \;\;&\;\;  -B
            \end{bmatrix}
            \kron \Id 
            } \p{\bx_\star, \bu_\star} = 0.
        \end{align*}
        However,~\eqref{eq:ABCD_null_condition} in Assumption~\ref{ass:ABCD_solution_and_fixed_point_condition} implies that
        \begin{align*}
            \p{
            \begin{bmatrix}
                \SumToZeroMat^{\top} C & \SumToZeroMat^{\top} D \\
                0 & \mathbf{1}^{\top}
            \end{bmatrix}
            \kron \Id 
            } \p{\bx_\star, \bu_\star} = 0,
        \end{align*}
        with the interpretation that the block row containing $\SumToZeroMat^{\top}$ is removed when $m=1$. In particular, note that this implies that 
        \begin{align*}
            y_{\star}^{(1)} = \ldots = y_{\star}^{(m)} = y_{\star}, \quad\text{ and }\quad \sum_{i=1}^{m}u_{\star}^{(i)} = 0,
        \end{align*}
        where $\bu_{\star}=\p{u_{\star}^{(1)},\ldots,u_{\star}^{(m)}}$ and $\by_\star = \p{y_{\star}^{(1)}, \ldots, y_{\star}^{(m)}}$, for some common value $y_{\star}\in\calH$, since $\by_{\star} = \p{C \kron \Id} \bx_{\star} + \p{D \kron \Id} \bu_{\star}$ (and $\p{\SumToZeroMat^{\top}\kron \Id}\by_{\star}=\p{y_{\star}^{(1)} - y_{\star}^{(m)}, \ldots, y_{\star}^{(m-1)}-y_{\star}^{(m)}}$ if $m>1$). In particular, $(\bx_\star,\bu_\star,\by_\star,\bFcn_\star)\in\Omega_{\textup{fixed-point encodings}}\p{f_1,\ldots,f_m}$.
        This proves the $\subseteq$ inclusion for~\eqref{eq:in_def_all_fps_are_fp_encodings}. 
        
    \textit{(ii) $\Rightarrow$ (i):} 
        We prove the contrapositive.
        
        First, suppose that~\eqref{eq:ABCD_range_condition} does not hold, i.e., there exists $\p{y_{\star},\hat{\bu}_{\star}}\in\calH\times\calH^{m-1}$ such that 
        \begin{align}\label{eq:in_proof_x_does_not_exist}
            \p{
            \begin{bmatrix}
                I-A \\
                -C 
            \end{bmatrix} \kron \Id
            } \bx_{\star}
            =
            \p{
            \begin{bmatrix}
                B\SumToZeroMat & 0 \\
                D\SumToZeroMat & -\mathbf{1} 
            \end{bmatrix} \kron \Id
            } \p{\hat{\bu}_{\star}, y_{\star}}
        \end{align}
        does not hold for any $\bx_{\star}\in\calH^{n}$. Define $\bu_{\star}=\p{u_{\star}^{(1)},\ldots,u_{\star}^{(m)}}\in\calH^{m}$ such that
        \begin{align*}
            \bu_{\star} =
            \begin{cases}
                \p{\SumToZeroMat\kron\Id} \hat{\bu}_{\star} & \text{if } m>1, \\
                0 & \text{if } m=1,
            \end{cases}
        \end{align*}
        and note that $\sum_{i=1}^{m}u_{\star}^{(i)} = 0$ holds by construction. Note that~\eqref{eq:in_proof_x_does_not_exist} then implies that  
        \begin{equation*}
            \begin{aligned} 
                \bx_{\star} &= \p{A \kron \Id} \bx_{\star} + \p{B \kron \Id} \bu_{\star}, \\
                \p{\mathbf{1} \kron \Id} y_{\star} &=  \p{C \kron \Id} \bx_{\star} + \p{D \kron \Id} \bu_{\star}
            \end{aligned} 
        \end{equation*}
        does not hold for any $\bx_{\star}\in\calH^{n}$. Thus, if we can show that there exists $\p{f_{1},\ldots,f_m}\in\prod_{i=1}^{m}\mathcal{F}_{\sigma_{i},\beta_{i}}$ such that $\bm{\partial f}\p{\p{\mathbf{1} \kron \Id} y_{\star}} = \set{\bu_{\star}}$ holds, then we are done since this shows that there exists $y_{\star} \in \zer \p{ \sum_{i=1}^m\partial f_i }$ such that the implication in~\eqref{eq:in_def_all_sols_have_fp_encodings} fails.
        Let $(\delta_{1},\ldots,\delta_{m})\in\prod_{i=1}^{m}[\sigma_i,\beta_i]$ and $f_{i}:\calH\rightarrow\reals$ such that 
        \begin{align*}
            f_{i}(y) = \frac{\delta_{i}}{2}\norm{y - y_{\star}}^{2} + \inner{u_{\star}^{(i)}}{y} 
        \end{align*}
        for each $y\in\calH$ and $i\in\llbracket1,m\rrbracket$. Then $\p{f_{1},\ldots,f_m}\in\prod_{i=1}^{m}\mathcal{F}_{\sigma_{i},\beta_{i}}$ is clear, and $\bm{\partial f}\p{\p{\mathbf{1} \kron \Id} y_{\star}} = \set{\bu_{\star}}$ holds since
        \begin{align*}
            \partial f_{i}(y_{\star}) = \set{u_{\star}^{(i)}}
        \end{align*}
        for each $i\in\llbracket1,m\rrbracket$.
        
        Second, suppose that~\eqref{eq:ABCD_null_condition} does not hold, i.e., there exists $\p{\bx_\star, \bu_\star}\in\calH^{n}\times\calH^{m}$ such that 
        \begin{align*}
            \p{
            \begin{bmatrix}
                I-A \;\;&\;\;  -B
            \end{bmatrix}
            \kron \Id 
            } \p{\bx_\star, \bu_\star} = 0,
        \end{align*}
        but
        \begin{align*}
            \p{
            \begin{bmatrix}
                \SumToZeroMat^{\top} C & \SumToZeroMat^{\top} D \\
                0 & \mathbf{1}^{\top}
            \end{bmatrix}
            \kron \Id 
            } \p{\bx_\star, \bu_\star} \neq 0.
        \end{align*}
        If we let $\bu_{\star}=\p{u_{\star}^{(1)},\ldots,u_{\star}^{(m)}}$ and $\by_{\star} = \p{y_{\star}^{(1)}, \ldots, y_{\star}^{(m)}} = \p{C \kron \Id} \bx_{\star} + \p{D \kron \Id} \bu_{\star}$, then either or both of $y_{\star}^{(1)} = \ldots = y_{\star}^{(m)}$ and $\sum_{i=1}^{m}u_{\star}^{(i)} = 0$ fail. Thus, if we can show that there exists $\p{f_{1},\ldots,f_m}\in\prod_{i=1}^{m}\mathcal{F}_{\sigma_{i},\beta_{i}}$ such that~\eqref{eq:fixed_point} holds, then we are done since this shows that there exists $(\bx_\star,\bu_\star,\by_\star,\bFcn_\star)\in\Omega_{\textup{fixed points}}\p{f_1,\ldots,f_m}$ such that $(\bx_\star,\bu_\star,\by_\star,\bFcn_\star)\notin\Omega_{\textup{fixed-point encodings}}\p{f_1,\ldots,f_m}$. Let $(\delta_{1},\ldots,\delta_{m})\in\prod_{i=1}^{m}[\sigma_i,\beta_i]$ and $f_{i}:\calH\rightarrow\reals$ such that 
        \begin{align*}
            f_{i}(y) = \frac{\delta_{i}}{2}\norm{y - y_{\star}^{(i)}}^{2} + \inner{u_{\star}^{(i)}}{y} 
        \end{align*}
        for each $y\in\calH$ and $i\in\llbracket1,m\rrbracket$. Then $\p{f_{1},\ldots,f_m}\in\prod_{i=1}^{m}\mathcal{F}_{\sigma_{i},\beta_{i}}$, and~\eqref{eq:fixed_point} holds since
        \begin{align*}
            \partial f_{i}(y_{\star}^{(i)}) = \set{u_{\star}^{(i)}}
        \end{align*}
        for each $i\in\llbracket1,m\rrbracket$.
\end{proof}

\begin{remark}
There exist many different choices of $A$, $B$, $C$, and $D$ in \eqref{eq:linear_system_with_nonlinearity} that can represent a given first-order method. The dimension $m$ in $\by\in\calH^m$ is fixed due to the number of functional components in problem~\eqref{eq:the_problem}, but the dimension $n$ in $\bx\in\calH^n$ can vary among representations. In fact, there exists a minimal $n$ such that a given first-order method can be represented as \eqref{eq:linear_system_with_nonlinearity}, leading to a {\emph{minimal representation}}. A necessary condition is that 
\begin{align}
    \rank\begin{bmatrix}I-A \;\;&\;\;  -B\end{bmatrix}=n\qquad{\hbox{and}}\qquad\rank\begin{bmatrix}I-A\\
    -C\end{bmatrix}=n,
    \label{eq:minimal_representation_rank_conditions}
\end{align}
where both matrices appear in Assumption~\ref{ass:ABCD_solution_and_fixed_point_condition}.
If these do not hold, the system is not {\emph{controllable}} (also often called reachable) \cite[Definitions~6.D1]{Chen2013Linear} or {\emph{observable}} \cite[Definitions~6.D2]{Chen2013Linear}, respectively. This implies that the representation is not minimal \cite[Theorem~25.2]{Lectures2023Dahleh} and that it is possible, for instance via a Kalman decomposition \cite[Section~25.2]{Lectures2023Dahleh}, to go from this non-minimal representation to a minimal representation that satisfies \eqref{eq:minimal_representation_rank_conditions} and represents the same algorithm.
\label{rem:rank_n}
\end{remark}

\begin{remark}
\Update{Previously, \cite{Michalowsky_Scherer_Ebenbauer_2021,Sundararajan_Van_Scoy_Lessard_2020} derived necessary and sufficient conditions for the existence of a fixed point from which a solution can be extracted, using algorithm representations different from~\eqref{eq:linear_system_with_nonlinearity}. Note that the existence of a fixed point from which a solution can be extracted differs from the concept of the fixed-point encoding property considered here.}
\end{remark}

\subsection{Well-posedness}\label{sec:well-posedness}

\Update{When analyzing existing algorithms, well-posedness is usually clear from the outset. However, when taking the more abstract point of view, as given by Algorithm~\eqref{eq:linear_system_with_nonlinearity}, further discussion is warranted.} We would like Algorithm~\eqref{eq:linear_system_with_nonlinearity} to be well-posed in the sense that it can be initiated at an arbitrary $\bx_{0}\in\calH^{n}$ and produce an infinite sequence $\{(\bx_k, \bu_k, \by_k, \bFcn_k)\}_{k=0}^{\infty}$ obeying the algorithm dynamics~\eqref{eq:linear_system_with_nonlinearity}. This holds if for each $\bx\in\calH^n$, there exist $\bu\in\calH^m$ and $\by\in\calH^m$ such that
\begin{equation}
\begin{aligned}
\by &= (C\kron\Id)\bx+(D\kron\Id)\bu\\
\bu &\in \bm\partial \bfcn(\by).
\end{aligned}
\label{eq:uy_system}
\end{equation}
In addition, if $\bu\in\calH^m$ and $\by\in\calH^m$ are unique, then the generated sequence is unique.

If $D$ has a lower-triangular structure, \eqref{eq:uy_system} can be solved using back-substitution. If $[D]_{i,i}\neq 0$, an implicit step is needed to find $y^{(i)}$ and $u^{(i)}$. If $[D]_{i,i}<0$, this implicit step is a proximal evaluation, which implies uniqueness. If $[D]_{i,i}=0$, $u^{(i)}$ is found via direct evaluation of $\partial f_i\p{y^{(i)}}$ which is always unique if $f_i$ is differentiable. 

\begin{assumption}\label{ass:well-posedness}
    Let 
    \begin{align*}
        I_{\textup{differentiable}} = \set{i\in\llbracket1,m\rrbracket: \beta_{i} < + \infty} \quad \text{ and } \quad I_{D} = \set{i\in\llbracket1,m\rrbracket: [D]_{i,i}< 0}
    \end{align*}
    and assume that $I_{\textup{differentiable}}\cup I_{D} = \llbracket1,m\rrbracket$
    and 
        $D$ is lower triangular with nonpositive diagonal elements.
\end{assumption}

The requirements in Assumption~\ref{ass:well-posedness} give rise to causal algorithms that generate unique and infinite sequences that evaluate either a proximal operator or a gradient for each $f_i$ and linearly combine results of previous evaluations to form inputs.

\begin{proposition}\label{prp:well-posedness}
    Suppose that Assumption~\ref{ass:well-posedness} holds. Then for any $\p{f_{1},\ldots,f_m}\in\prod_{i=1}^{m}\mathcal{F}_{\sigma_{i},\beta_{i}}$ and $\bx_{0}=\p{x_{0}^{(1)},\ldots,x_{0}^{(n)}}\in\calH^{n}$, algorithm~\eqref{eq:linear_system_with_nonlinearity} produces a unique sequence $\{(\bx_k, \bu_k, \by_k, \bFcn_k)\}_{k=0}^{\infty}$ obeying the algorithm dynamics~\eqref{eq:linear_system_with_nonlinearity} and can be implemented as the following causal procedure:
    \begin{align}\label{eq:explicit_ABCD_update}
        \begin{aligned}
            &\text{for } k = 0,1,\ldots \\
             &\left\lfloor
            \begin{aligned}
                &\text{for } i = 1,\ldots,m \\
                &\left\lfloor
                \begin{aligned}
                &v_k^{(i)} = \sum_{j=1}^{n}[C]_{i,j}x_{k}^{(j)} + \sum_{j=1}^{i-1}[D]_{i,j}u_{k}^{(j)}, \\
                &y^{(i)}_k = 
                \begin{cases}
                    \prox_{-[D]_{i,i}f_{i}}\p{v_k^{(i)}} & \text{if } i \in I_D, \\
                     v_k^{(i)} & \text{if } i \notin I_D,
                \end{cases} \\
                &u^{(i)}_k = 
                \begin{cases}
                    (-[D]_{i,i})^{-1}\p{v_{k}^{(i)}-y_{k}^{(i)}} & \text{if } i \in I_{D}, \\
                     \nabla f_{i}\p{y_{k}^{(i)}} & \text{if } i \notin I_{D},
                \end{cases} \\
                &F_{k}^{(i)} = f_{i}\p{y_{k}^{(i)}},
                \end{aligned}\right. \\
                &\bx_{k+1} =\p{x_{k+1}^{(1)},\ldots,x_{k+1}^{(n)}} = \p{A \kron \Id} \bx_{k} + \p{B \kron \Id} \bu_{k},\\
            \end{aligned}\right.
        \end{aligned}
    \end{align}
    where $\bu_{k}=\p{u_{k}^{(1)},\ldots,u_{k}^{(m)}}$, $\by_{k}=\p{y_{k}^{(1)},\ldots,y_{k}^{(m)}}$, $\bFcn_{k}=\p{F_{k}^{(1)},\ldots,F_{k}^{(m)}}$ and the empty sum is set equal to zero by convention.
\end{proposition}
\begin{proof}
    Let $\p{f_{1},\ldots,f_m}\in\prod_{i=1}^{m}\mathcal{F}_{\sigma_{i},\beta_{i}}$. Consider an arbitrary $k\in\naturals$ and pick any $\bx_{k}=\p{x_{k}^{(1)},\ldots,x_{k}^{(n)}}\in\calH^{n}$. For $i\in\llbracket1,m\rrbracket$ in ascending order:
    \begin{itemize}
        \item  $v_k^{(i)}$ in the inner loop in~\eqref{eq:explicit_ABCD_update} is a linear combination of previously calculated/known quantities.  
        \item If $i\in I_{D}$, then~\eqref{eq:linear_system_with_nonlinearity} and the structure of $D$ in Assumption~\ref{ass:well-posedness} give that 
        \begin{align*}
            y^{(i)}_k \in v_k^{(i)} + [D]_{i,i} \partial f_{i}\p{y^{(i)}_k}&\quad\Leftrightarrow\quad\underbrace{(-[D]_{i,i})^{-1}\p{v_{k}^{(i)}-y_{k}^{(i)}}}_{=\, u_k^{(i)}} \in \partial f_{i}\p{y^{(i)}_k}\\
            &\quad\Leftrightarrow\quad y^{(i)}_k = \prox_{-[D]_{i,i}f_{i}}\p{v_k^{(i)}},
        \end{align*}
        which is unique since the proximal operator is single-valued with full domain under our assumptions (recall that each $f_{i}$ is assumed to be proper, lower semicontinuous, and convex).
        \item If $i\notin I_D$, then $f_i$ is differentiable due to Assumption~\ref{ass:well-posedness}, and~\eqref{eq:linear_system_with_nonlinearity} gives that $y^{(i)}_k = v^{(i)}_k$ and $u^{(i)}_k = \nabla f_{i}\p{y_{k}^{(i)}}$.
    \end{itemize}
    An inductive argument concludes the proof.
\end{proof}

The requirement that $D$ is lower triangular is for convenience. If there exists a permutation $\pi:\llbracket 1,m\rrbracket\to\llbracket 1,m\rrbracket$ with associated permutation matrix $P_\pi$ such that $P_\pi D P_{\pi}^{\top}$ is lower triangular, the resulting algorithm is equivalent to \eqref{eq:explicit_ABCD_update}. Let $\bar{\by}_k=P_\pi\by_k$, $\bar{\bu}_k=P_\pi\bu_k$, and $\bar{\bfcn}=\bfcn\circ (P_\pi^{\top}\kron\Id)$ (that just reorders the inputs). Then $\bm\partial\bar{\bfcn} = (P_\pi\kron\Id)\circ\bm\partial\bfcn\circ (P_\pi^{\top}\kron\Id)$ and the algorithm is equivalent to
\begin{equation*}
\begin{aligned}
                \bx_{k+1} &= \p{A \kron \Id} \bx_{k} + \p{BP_\pi^{\top} \kron \Id} \bar{\bu}_{k}, \\
                \bar{\by}_{k} &=  \p{P_\pi C \kron \Id} \bx_{k} + \p{P_\pi DP_\pi^{\top} \kron \Id} \bar{\bu}_{k}, \\
                \bar{\bu}_{k} &\in \bm{\partial} \bar{\bfcn}(\bar{\by}_{k}),
            \end{aligned}
\end{equation*}
which can be implemented as in \eqref{eq:explicit_ABCD_update}.

If no permutation matrix exists such that $P_\pi DP_\pi^{\top}$ is lower triangular, then there exist $i<j$ with $i,j\in\llbracket 1,m\rrbracket$ such that $[D]_{i,j}\neq 0$ and $[D]_{j,i}\neq 0$. This couples the $\partial f_i$ and $\partial f_j$ evaluations such that back-substitution fails and these updates cannot in general be done using only proximal operator or gradient evaluations of $f_i$ and $f_j$ individually.

Since the linear combinations decided by $A$, $B$, $C$, and $D$ are arbitrary, all first-order methods that use fixed linear combinations of previously computed quantities and evaluate each individual subdifferentials only once per iteration and either via a proximal operator or gradient evaluation can be implemented as in \eqref{eq:explicit_ABCD_update}, potentially after a permutation of variables. We provide a list of examples in Section~\ref{sec:alg_examples} that all satisfy Assumption~\ref{ass:well-posedness}. They also satisfy Assumption~\ref{ass:ABCD_solution_and_fixed_point_condition}, implying that solving \eqref{eq:the_problem_inclusion} is equivalent to finding a fixed point of the algorithm, and the rank conditions in \eqref{eq:minimal_representation_rank_conditions}.

\subsection{Examples}
\label{sec:alg_examples}

In this section, we provide examples of a few well-known algorithms that can be written as~\eqref{eq:linear_system_with_nonlinearity}.

\subsubsection{Douglas--Rachford method} 
\label{sec:douglas_rachford}
Let $\gamma \in \reals_{++}$, $\lambda\in\reals\setminus\{0\}$ and $f_{1},f_{2}\in\mathcal{F}_{0,\infty}$. The Douglas--Rachford method \cite{Douglas_Rachford_numerical_1956,Eckstein_Bertsekas_DouglasRachfordSplitting_1992,Lions_Mercier_SplittingAlgorithmsSum_1979} is given by
\begin{equation*}
    \begin{aligned}
    y_{k}^{(1)}&=\prox_{\gamma f_1}\p{x_k},\\
    y_{k}^{(2)}&=\prox_{\gamma f_2}\p{2y_k^{(1)}-x_k},\\
    x_{k+1}&=x_k+\lambda\p{y_k^{(2)}-y_k^{(1)}},
    \end{aligned}
\end{equation*}
which can equivalently be written as
\begin{equation*}
    \begin{aligned}
        \gamma^{-1}\p{x_k-y_k^{(1)}}&\in\partial f_1\p{y_k^{(1)}},\\
        \gamma^{-1}\p{\p{2y^{(1)}_{k}-x_k}-y_k^{(2)}}&\in\partial f_2\p{y_k^{(2)}},\\
        x_{k+1}&=x_k+\lambda\p{y_k^{(2)}-y_k^{(1)}}.
    \end{aligned}
\end{equation*}
By letting $\bx_{k}=x_{k}$, $\by_{k}=\p{y_k^{(1)},y_k^{(2)}}$ and $\bu_{k}=(\gamma^{-1}(x_k-y_k^{(1)}),\gamma^{-1}(2y_k^{(1)}-x_k-y_k^{(2)}))$, we get
\begin{align*}
\bx_{k+1}&=\bx_k+\left(\begin{bmatrix}-\gamma\lambda&-\gamma\lambda\end{bmatrix}\kron\Id\right)\bu_{k},\\
    \by_{k}&=\left(\begin{bmatrix} 1\\1\end{bmatrix}\kron\Id\right)\bx_k+\left(\begin{bmatrix}-\gamma&0\\-2\gamma&-\gamma\end{bmatrix}\kron\Id\right)\bu_{k},\\
    \bu_{k}&\in\bm\partial\bfcn(\by_{k}),
\end{align*}
where $\bm{\partial}\bfcn(\by)=\partial f_1\p{y^{(1)}}\times\partial f_2\p{y^{(2)}}$ for each $\by=\p{y^{(1)},y^{(2)}}\in\calH^2$, which matches the form~\eqref{eq:linear_system_with_nonlinearity}.

\subsubsection{Gradient method with heavy-ball momentum}
\label{sec:gradient_method_heavy-ball_momentum}
Let $\gamma,\beta_1 \in \reals_{++}$, $\delta\in\reals$ and $f_{1}\in\mathcal{F}_{0,\beta_1}$. The gradient method with heavy-ball momentum is given by
\begin{align*}
    x_{k+1} = x_k-\gamma\nabla f_1(x_k)+\delta(x_k-x_{k-1}).
\end{align*}
By letting $\bx_{k}=(x_{k},x_{k-1})$, $\by_k=x_{k}$, and $\bu_k=\nabla f_1(x_k)$, we get
\begin{equation*}
    \begin{aligned}
    \bx_{k+1}&=\left(\begin{bmatrix}1+\delta&-\delta\\1&0\end{bmatrix}\otimes \Id\right)\bx_{k}+\left(\begin{bmatrix}-\gamma\\0 \end{bmatrix}\otimes\Id\right)\bu_k,\\
    \by_k&=\left(\begin{bmatrix}1&0\end{bmatrix}\otimes \Id\right)\bx_{k}, \\
    \bu_{k} &\in \bm{\partial}\bfcn(\by_{k}),
    \end{aligned}
\end{equation*}
where $\bm{\partial}\bfcn(y)=\{\nabla f_1(y)\}$ for each $y\in\calH$, which matches the form~\eqref{eq:linear_system_with_nonlinearity}. 

\subsubsection{Proximal gradient method with heavy-ball momentum terms} 
\label{sec:proximal_gradient_method_heavy-ball_momentum}
Let $\gamma,\beta_1 \in \reals_{++}$, $\delta_1,\delta_2\in\reals$, $f_{1}\in\mathcal{F}_{0,\beta_1}$ and $f_{2}\in\mathcal{F}_{0,\infty}$. A proximal gradient method with heavy-ball momentum terms is given by
\begin{align*}
    x_{k+1} = \prox_{\gamma f_2}\p{x_k-\gamma\nabla f_1(x_k)+\delta_1(x_k-x_{k-1})}+\delta_2(x_k-x_{k-1}).
\end{align*}
By letting  $\bx_{k}=(x_{k},x_{k-1})$, $\by_{k}=\p{x_k,x_{k+1}-\delta_2(x_k-x_{k-1})}$, $\bu_{k}=(\nabla f_1(x_k),\gamma^{-1}(x_k-\gamma \nabla f_1(x_k)+(\delta_1+\delta_2)(x_k-x_{k-1})-x_{k+1}))$, we get
\begin{equation*}
    \begin{aligned}
    \bx_{k+1}   &=  
    \left(\begin{bmatrix}1+\delta_1+\delta_2&-\delta_1-\delta_2\\1&0\end{bmatrix}\otimes \Id\right)\bx_{k} + 
    \left(\begin{bmatrix}-\gamma&-\gamma\\0&0 \end{bmatrix}\otimes\Id\right)\bu_{k} \\
    \by_{k} &=    
    \left(\begin{bmatrix}1&0\\1+\delta_1&-\delta_1\end{bmatrix}\otimes \Id\right)\bx_{k} + 
    \left(\begin{bmatrix}0&0\\-\gamma&-\gamma \end{bmatrix}\otimes\Id\right)\bu_{k}, \\
    \bu_{k}&\in\bm\partial\bfcn(\by_{k}),
    \end{aligned}
\end{equation*}
where $\bm{\partial}\bfcn(\by)=\set{\nabla f_{1}\p{y^{(1)}}}\times\partial f_2\p{y^{(2)}}$ for each $\by=\p{y^{(1)},y^{(2)}}\in\calH^2$, which matches the form~\eqref{eq:linear_system_with_nonlinearity}.

\subsubsection{Davis--Yin three-operator splitting method}
\label{sec:davis_yin}
Let $\gamma,\lambda \in \reals_{++}$, $0 \leq \sigma_i < \beta_i \leq +\infty$ and $f_i\in\mathcal{F}_{\sigma_i,\beta_i}$ for each $i\in\llbracket1,3\rrbracket$, and $\beta_{2}<\infty$. The three-operator splitting method by Davis and Yin in~\cite{Davis2017AThreeOperator} is given by
\begin{align*}
    x_k &= \prox_{\gamma f_1}\p{z_k}, \\
    z_{k+\frac{1}{2}} &= 2x_k - z_k - \gamma \nabla f_2(x_k), \\
    z_{k+1} &= z_k + \lambda \p{\prox_{\gamma f_3}\p{z_{k+\frac{1}{2}}} - x_k}.
\end{align*}
\begin{sloppypar}By letting $\bx_{k} = z_k$, $\by_{k}=(x_k,x_k,x_k+\lambda^{-1}(z_{k+1}-z_k))$ and $\bu_{k}= (\gamma^{-1}(z_k-x_k),\gamma^{-1}(2x_k - z_k - z_{k + \frac{1}{2}}),\gamma^{-1}(z_{k+\frac{1}{2}}-x_k-\lambda^{-1}(z_{k+1}-z_{k})))$, we get
\end{sloppypar}
\begin{equation*}
    \begin{aligned}
    \bx_{k+1}   &= \bx_{k} + 
    \left(\begin{bmatrix}
        -\gamma\lambda & -\gamma\lambda & -\gamma\lambda 
    \end{bmatrix}\otimes\Id\right)\bu_{k}, \\
    \by_{k} &=    
    \left(\begin{bmatrix}
        1 \\
        1 \\
        1 
    \end{bmatrix}\otimes \Id\right)\bx_{k} + 
    \left(\begin{bmatrix}
        -\gamma & 0 & 0 \\
        -\gamma & 0 & 0 \\
        -2\gamma & -\gamma & -\gamma \\
    \end{bmatrix}\otimes\Id\right)\bu_{k}, \\
    \bu_{k}&\in\bm\partial\bfcn(\by_{k}),
    \end{aligned}
\end{equation*}
where $\bm{\partial}\bfcn(\by)=\partial f_1\p{y^{(1)}}\times\set{\nabla f_{2}\p{y^{(2)}}}\times\partial f_3\p{y^{(3)}}$ for each $\by=\p{y^{(1)},y^{(2)},y^{(3)}}\in\calH^3$, which matches the form~\eqref{eq:linear_system_with_nonlinearity}.

\subsubsection{Chambolle--Pock method}
\label{sec:chambolle_pock}
Let $\tau_1, \tau_2 \in \reals_{++}$, $ \theta \in\reals$, $0 \leq \sigma_i < \beta_i \leq +\infty$ and $f_i\in\mathcal{F}_{\sigma_i,\beta_i}$ for each $i\in\llbracket1,2\rrbracket$. The method by Chambolle and Pock in~\cite[Algorithm~1]{Chambolle_Pock_ergodic_2011} is given by
\begin{align*}
x_{k+1} &= \prox_{\tau_1 f_1} \p{x_k - \tau_1 y_{k}}, \\
y_{k+1} &= \prox_{\tau_2 f_2^*} \p{y_{k} + \tau_2 \left(x_{k+1} + \theta \p{x_{k+1} - x_k}\right)}.
\end{align*}
By letting $\bx_{k} = \p{x_k, y_k}$, $\by_{k}=( x_{k+1},\frac{1}{\tau_2} \p{y_k - y_{k+1}} + \p{1 + \theta} x_{k+1} - \theta x_{k})$, and $\bu_{k}= (\frac{1}{\tau_1} \p{x_k - x_{k+1}} - y_{k},y_{k+1})$, we get
\begin{equation*}
    \begin{aligned}
    \bx_{k+1}   &= 
    \left(\begin{bmatrix}
    1 & -\tau_1 \\
    0 & 0 \\
    \end{bmatrix}\otimes\Id\right)\bx_{k} + 
    \left(\begin{bmatrix}
        -\tau_1 & 0 \\
        0 & 1
    \end{bmatrix}\otimes\Id\right)\bu_{k}, \\
    \by_{k} &=    
    \left(\begin{bmatrix}
       1 & -\tau_1 \\
        1 & \frac{1}{\tau_2}-\tau_1(1+\theta)
    \end{bmatrix}\otimes \Id\right)\bx_{k} + 
    \left(\begin{bmatrix}
        -\tau_1 & 0 \\
        -\tau_1(1+\theta) & - \frac{1}{\tau_2}
    \end{bmatrix}\otimes\Id\right)\bu_{k}, \\
    \bu_{k}&\in\bm\partial\bfcn(\by_{k}),
    \end{aligned}
\end{equation*}
where $\bm{\partial}\bfcn(\by)=\partial f_1\p{y^{(1)}}\times\partial f_2\p{y^{(2)}}$ for each $\by=\p{y^{(1)},y^{(2)}}\in\calH^2$, which matches the form~\eqref{eq:linear_system_with_nonlinearity}.  

\section{Interpolation}\label{sec:interpolation}

Tightness of our methodology hinges critically on so-called \emph{interpolation conditions} for function classes that have been developed in the PEP literature~\cite{Taylor_Hendrickx_Glineur_2017a,Taylor_Hendrickx_Glineur_2017b}. The following theorem is proved in \cite[Theorem 4]{Taylor_Hendrickx_Glineur_2017a}.

\begin{theorem}\label{thm:interpolation}
Let $0\leq\sigma<\beta\leq +\infty$ and $\{(y_{i},F_{i},u_{i})\}_{i\in \mathcal{I}}$ be a finite family of triplets in $\calH\times\reals\times\calH$ indexed by $\mathcal{I}$. 
Then the following are equivalent: 
\begin{itemize}
    \item[(i)] There exists $f\in\mathcal{F}_{\sigma,\beta}$ such that 
\begin{align*}
    f(y_i) = F_i  \text{ and }  u_i \in \partial f (y_i)
\end{align*}
for each $i\in\mathcal{I}$.
\item[(ii)] It holds that 
\begin{align*}
    F_{i} \geq F_{j} + \inner{u_{j}}{y_{i}-y_{j}} + \frac{\sigma}{2}\norm{y_i-y_j}^{2}+ \frac{1}{2(\beta-\sigma)}\norm{u_i-u_j-\sigma(y_i-y_j)}^2
\end{align*}
for each $i,j\in\mathcal{I}$, where $\frac{1}{2(\beta-\sigma)}$ is interpreted as $0$ in the case $\beta = +\infty$.
\end{itemize}
\end{theorem}

Next, we adapt these interpolation conditions to our framework. In the following, we let $ \mathbfcal{F}_{\bm \sigma, \bm \beta}$ denote the class of all mappings $\bfcn:\calH^m\to\p{\reals\cup\{+\infty\}}^m$ defined by~\eqref{eq:product_f} for every possible choice of $f_i\in \mathcal{F}_{\sigma_i,\beta_i}$ and $i\in\llbracket1,m\rrbracket$. Moreover, with each $\bfcn \in \mathbfcal{F}_{\bm \sigma, \bm \beta}$, we associate the mapping $\bm{\partial}\bfcn:\calH^m\to 2^{\calH^m}$ defined by~\eqref{eq:product_f_subdiff}.

\begin{corollary}\label{cor:interpolation}
Let 
\begin{align}\label{eq:M_l}
    \bM_l &=
    \begin{cases}
    \dfrac{1}{2(\beta_{l}-\sigma_{l})}\begin{bmatrix}
     \beta_{l}\sigma_{l} & -\sigma_{l} & \beta_{l}\\
      -\sigma_{l} & 1 & -1 \\
     \beta_{l} & -1 & 1 
    \end{bmatrix}\kron \diag\p{e_l} & \text{if } \beta_{l} < \infty,\\
    \dfrac{1}{2}\begin{bmatrix}
     \sigma_{l} & 0 & 1\\
      0 & 0 & 0 \\
     1 & 0 & 0 
    \end{bmatrix}\kron \diag\p{e_l} & \text{if } \beta_{l} = +\infty,
    \end{cases} \\\label{eq:a_l}
    \ba_l &= -e_l
\end{align}
for each $l\in\llbracket1,m\rrbracket$, where $\kron$ denotes the Kronecker product and $\set{e_i}_{i=1}^{m}$ denotes the standard basis vectors of $\reals^{m}$. Then for each finite family of triplets in $\calH^m\times\reals^m\times\calH^m$ indexed by $\mathcal{I}$, $\{(\by_{i},\bFcn_{i},\bu_{i})\}_{i\in \mathcal{I}}$, the following are equivalent: 
\begin{itemize}
    \item[(i)] There exist $\bfcn\in \mathbfcal{F}_{\bm \sigma, \bm \beta}$ such that 
    \begin{align*}
        \bfcn(\by_i) = \bFcn_{i} \text{ and }  \bu_{i}\in\bm{\partial}\bfcn(\by_i)
    \end{align*}
    for each $i\in\mathcal{I}$.
    \item[(ii)] It holds that 
    \begin{align*}
        \ba_l^{\top}(\bFcn_{i}-\bFcn_{j}) + \mathcal{Q}\p{\bM_{l},(\by_i-\by_j,\bu_i,\bu_j)} \leq 0
    \end{align*}
    for each $i,j\in\mathcal{I}$ and $l\in\llbracket1,m\rrbracket$.
\end{itemize}
Moreover,
\begin{align*}
    \mathcal{Q}\p{\bM_{l},(0,\bu,\bu)}=0
\end{align*}
for each $\bu\in\calH^{m}$ and  $l\in\llbracket1,m\rrbracket$.
\end{corollary}  

\section{Lyapunov inequalities}\label{sec:Lyapunov}
Convergence properties of many first-order methods can be analyzed via so-called {\emph{Lyapunov inequalities}}. We consider Lyapunov inequalities of the form
\begin{align}\label{eq:Lyapunov_inequality}
V(\bxi_{k+1},\bxi_\star)\leq\rho V(\bxi_k,\bxi_\star)-R(\bxi_k,\bxi_\star),
\end{align}
\begin{sloppypar}
where $\rho\in[0,1]$, $\bxi_k=(\bx_k,\bu_k,\by_k,\bFcn_k)\in\mathcal{S}$ contains all algorithm variables in iteration $k$, $\bxi_{k+1}=(\bx_{k+1},\bu_{k+1},\by_{k+1},\bFcn_{k+1})\in\mathcal{S}$ contains all algorithm variables in iteration $k+1$, $\bxi_\star=(\bx_{\star},\bu_{\star},\by_{\star},\bFcn_{\star})\in \mathcal{S}$ is a fixed point, $V:\mathcal{S}\times\mathcal{S}\to\reals$ is called a {\emph{Lyapunov function}}, $R:\mathcal{S}\times\mathcal{S}\to\reals$ is called a {\emph{residual function}}, and $\mathcal{S}=\calH^{n}\times\calH^{m}\times\calH^{m}\times\reals^m$. Once such an inequality has been established, various convergence properties may be concluded depending on the properties of the functions $V$ and $R$. 
\end{sloppypar}

We consider quadratic ansatzes of the functions $V$ and $R$ given by
\begin{align}\label{eq:Lyapunov_form}
    V(\bxi,\bxi_{\star}) &=\quadform{Q}{(\bx-\bx_{\star},\bu,\bu_{\star})}+q^{\top}(\bFcn-\bFcn_{\star}),\\ \label{eq:Residual_form}
    R(\bxi,\bxi_{\star}) &=\quadform{S}{(\bx-\bx_{\star},\bu,\bu_{\star})}+s^{\top}(\bFcn-\bFcn_{\star})
\end{align}
for each $\bxi,\bxi_\star\in\mathcal{S}$, respectively, where $Q,S\in\sym^{n+2m}$ and $q,s\in\reals^{m}$ parameterize the functions. These are general ansatzes that allow for arbitrary linear combinations of scalar products between linear combinations of $x^{(i)}-x_\star^{(i)}$, $u^{(i)}$, and $u_\star^{(i)}$ and linear combinations of function-value differences $f_i\p{y^{(i)}}-f_i\p{y_\star^{(i)}}$.

To draw useful convergence conclusions from \eqref{eq:Lyapunov_inequality}, we enforce nonnegative quadratic lower bounds on $V$ and $R$ given by
\begin{align}\label{eq:Lyapunov_form_lower_bound}
    V(\bxi,\bxi_{\star}) &\geq\quadform{P}{(\bx-\bx_{\star},\bu,\bu_{\star})}+p^{\top}(\bFcn-\bFcn_{\star})\geq 0,\\ \label{eq:Residual_form_lower_bound}
    R(\bxi,\bxi_{\star}) &\geq\quadform{T}{(\bx-\bx_{\star},\bu,\bu_{\star})}+t^{\top}(\bFcn-\bFcn_{\star}) \geq 0,
\end{align}
where $P,T\in\sym^{n+2m}$ and $p,t\in\reals^{m}$. We do not enforce these inequalities on all of $\mathcal{S}\times\mathcal{S}$ but only when the first argument is a so-called {\emph{algorithm-consistent}} point and the second argument satisfies the fixed-point equations \eqref{eq:fixed_point}.
\begin{definition}[\normalfont Algorithm consistency] \label{def:algorithm_consistency}
    Consider algorithm~\eqref{eq:linear_system_with_nonlinearity}. The point $\bxi=(\bx,\bu,\by,\bFcn)\in\mathcal{S}$ is called {\emph{algorithm-consistent for $\bfcn\in \mathbfcal{F}_{\bm \sigma, \bm \beta}$}} if 
    \begin{align*}
        \by &= (C\kron\Id) \bx + (D\kron\Id) \bu, \\
        \bu &\in \bm{\partial}\bfcn(\by), \\
        \bFcn&= \bfcn(\by).
    \end{align*}
\end{definition}
To restrict~\eqref{eq:Lyapunov_form_lower_bound} and~\eqref{eq:Residual_form_lower_bound} on this subset of $\mathcal{S}\times\mathcal{S}$ gives a larger class of Lyapunov functions and residual functions compared to requiring them to hold on all of $\mathcal{S}\times\mathcal{S}$.

In the proposed methodology, the user specifies $(P,p,T,t,\rho)$ and the methodology provides $(Q,q,S,s)$ complying with \eqref{eq:Lyapunov_inequality}, \eqref{eq:Lyapunov_form_lower_bound}, and \eqref{eq:Residual_form_lower_bound}, if it exists.
When such a $(Q,q,S,s)$ exists, the choice of $(P,p,T,t,\rho)$ decides which convergence properties the analysis implies.
\begin{enumerate}[(i)]
    \item Suppose that $\rho\in[0,1[$. Then
    \begin{align*}
        0\leq\quadform{P}{(\bx_k-\bx_{\star},\bu_k,\bu_{\star})}+p^{\top}(\bFcn_{k}-\bFcn_{\star})\leq V(\bxi_k,\bxi_\star)\leq\rho^k V(\bxi_0,\bxi_\star)\to 0
    \end{align*}
    as $k\to\infty$. In particular, 
    \begin{align}\label{eq:V_lower_bound_linear_convergence}
        \set{\quadform{P}{(\bx_k-\bx_{\star},\bu_k,\bu_{\star})}+p^{\top}(\bFcn_k-\bFcn_{\star})}_{k\in\naturals}\quad \text{converges }\rho\text{-linearly to zero.}
    \end{align}
    \item Suppose that $\rho=1$. Then 
    \begin{align}\label{eq:R_lower_bound_summable}
        \sum_{k=0}^{\infty}\p{\quadform{T}{(\bx_k-\bx_{\star},\bu_k,\bu_{\star})}+t^{\top}(\bFcn_k-\bFcn_{\star})}\leq\sum_{k=0}^{\infty}R(\bxi_k,\bxi_\star)\leq V(\bxi_0,\bxi_\star),
    \end{align}
    using a telescoping summation argument.
    In particular, 
    \begin{align}\label{eq:R_lower_bound_converges}
        \set{\quadform{T}{(\bx_k-\bx_{\star},\bu_k,\bu_{\star})}+t^{\top}(\bFcn_k-\bFcn_{\star})}_{k\in\naturals}\quad \text{is summable and converges to zero.}
    \end{align}
    
\end{enumerate}
Therefore, $(P,p,T,t,\rho)$ needs to be chosen to extract interesting convergence results from the lower bounds. If $P=T=0$ and $p=t=0$, then $V$ and $R$ equal to the zero function gives a valid Lyapunov inequality~\eqref{eq:Lyapunov_inequality} that complies with the lower bounds~\eqref{eq:Lyapunov_form_lower_bound} and~\eqref{eq:Residual_form_lower_bound}, but is of no interest. Useful choices of $(P,p,T,t,\rho)$ that imply different specific convergence results are provided in Section~\ref{sec:lower_bound_choices}.

The above requirements on the Lyapunov inequality, the Lyapunov function, and the residual function are formalized in Definition~\ref{def:Lyapunov} after we define the notion of a {\emph{successor}}.

\begin{definition}[\normalfont Successor] 
    Consider algorithm~\eqref{eq:linear_system_with_nonlinearity}. Given an algorithm-consistent point $\bxi$ for some $\bfcn\in \mathbfcal{F}_{\bm \sigma, \bm \beta}$, we define a \emph{successor of $\bxi$} to be any point $\bxi_{+}=(\bx_{+},\bu_{+},\by_{+},\bFcn_{+})\in\mathcal{S}$ such that 
    \begin{align*}
        \bx_{+} &= (A\kron\Id) \bx + (B\kron\Id) \bu, \\
        \by_{+} &= (C\kron\Id) \bx_{+} + (D\kron\Id) \bu_{+}, \\
        \bu_{+} &\in \bm{\partial}\bfcn(\by_{+}), \\
        \bFcn_{+}&= \bfcn(\by_{+}).
    \end{align*}
\end{definition}

\begin{definition}[\normalfont Quadratic Lyapunov inequality]\label{def:Lyapunov}
    Let $V:\mathcal{S}\times\mathcal{S}\to \reals$ as in~\eqref{eq:Lyapunov_form}, $R:\mathcal{S}\times\mathcal{S}\to \reals$ as in~\eqref{eq:Residual_form}, $P,T\in\sym^{n+2m}$, $p,t\in\reals^{m}$ and $\rho\in[0,1]$. 
    We say that $V$ and $R$ satisfy the {\emph{$\p{P,p,T,t,\rho}$-quadratic Lyapunov inequality}} for algorithm~\eqref{eq:linear_system_with_nonlinearity} over the class $ \mathbfcal{F}_{\bm \sigma, \bm \beta}$ if:
    \begin{enumerate}[C1.]
        \item $V(\bxi_{+},\bxi_{\star}) \leq \rho V(\bxi,\bxi_{\star})-R(\bxi,\bxi_{\star})$ for each $\bxi\in \mathcal{S}$ that is algorithm-consistent for $\bfcn$, each successor $\bxi_{+}\in \mathcal{S}$ of $\bxi$, each $\bxi_{\star}\in \mathcal{S}$ that satisfies~\eqref{eq:fixed_point}, and each $\bfcn\in \mathbfcal{F}_{\bm \sigma, \bm \beta}$.\label{item:Lyapunov_inequality}
        \item $V(\bxi,\bxi_{\star})\geq \quadform{P}{(\bx-\bx_{\star},\bu,\bu_{\star})}+p^{\top}(\bFcn-\bFcn_\star)\geq 0$ for each $\bxi\in \mathcal{S}$ that is algorithm-consistent for $\bfcn$, each $\bxi_{\star}\in \mathcal{S}$ that satisfies~\eqref{eq:fixed_point}, and each $\bfcn\in \mathbfcal{F}_{\bm \sigma, \bm \beta}$. \label{item:V_lower_bound}
        \item $R(\bxi,\bxi_{\star})\geq \quadform{T}{(\bx-\bx_{\star},\bu,\bu_{\star})}+t^{\top}(\bFcn-\bFcn_\star)\geq 0$ for each $\bxi\in \mathcal{S}$ that is algorithm-consistent for $\bfcn$, each $\bxi_{\star}\in \mathcal{S}$ that satisfies~\eqref{eq:fixed_point}, and each $\bfcn\in \mathbfcal{F}_{\bm \sigma, \bm \beta}$. \label{item:R_lower_bound}
    \end{enumerate}
\end{definition}

The main result in Section~\ref{sec:main} is a necessary and sufficient condition for the existence of a $(P,p,T,t,\rho)$-quadratic Lyapunov inequality expressed as a semidefinite feasibility problem over the Lyapunov function and residual function parameters $(Q,q,S,s)$. This is done by providing a necessary and sufficient condition for each of C\ref{item:Lyapunov_inequality}, C\ref{item:V_lower_bound}, and C\ref{item:R_lower_bound}. Conditions C\ref{item:Lyapunov_inequality}, C\ref{item:V_lower_bound}, and C\ref{item:R_lower_bound} can all be stated as the verification of a quadratic function $\Phi:\mathcal{S}\times\mathcal{S}\to\reals$ to be nonpositive over the subset of $\mathcal{S}\times\mathcal{S}$ that includes algorithm-consistent points in the first argument and fixed points in the second. Restricting to this subset adds significant technical complication compared to verifying nonpositivity over the entirety of $\mathcal{S}\times\mathcal{S}$, but provides the added benefit of a more general Lyapunov analysis.

\subsection{Lower bounds and convergence implications}
\label{sec:lower_bound_choices}

In this section, we provide a few choices of $(P,p,T,t,\rho)$ from which we can draw specific convergence results, under the assumption that there exists a Lyapunov function $V$ and a residual function $R$ that satisfy the $\p{P,p,T,t,\rho}$-quadratic Lyapunov inequality. Moreover, we assume that Assumptions~\ref{ass:ABCD_solution_and_fixed_point_condition} and~\ref{ass:well-posedness} hold.

\subsubsection{Linear convergence of the distance to the solution}
Suppose that $\rho\in[0,1[$ and
\begin{align}\label{eq:linear_rate_distance_to_the_solution}
    \p{P,p,T,t} = \p{\begin{bmatrix} C & D & -D \end{bmatrix}^{\top}e_{i}e_{i}^{\top}\begin{bmatrix} C & D & -D \end{bmatrix}, 0,0,0}
\end{align}
for some $i\in\llbracket1,m\rrbracket$, where $\set{e_i}_{i=1}^{m}$ denotes the standard basis vectors of $\reals^{m}$. Then~\eqref{eq:V_lower_bound_linear_convergence} implies that the {\emph{squared distance to the solution}} $\set{\norm{y^{(i)}_{k} - y_{\star}}^{2}}_{k\in\naturals}$ converges $\rho$-linearly to zero, where $y_{\star}$ is the solution to~\eqref{eq:the_problem_inclusion}, since 
\begin{align*}
    \quadform{P}{(\bx_k-\bx_\star,\bu_k,\bu_\star)} + p^{\top}(\bFcn_k-\bFcn_{\star}) =\norm{y^{(i)}_{k}-y_\star}^2\geq 0.
\end{align*}
\Update{Note that we exclude the case $\rho = 1$ since we can only guarantee that the squared distance to the solution remains bounded but not necessarily linearly convergent. }

\subsubsection{$\mathcal{O}\p{1/k}$ ergodic convergence}
Suppose that $\rho = 1$. 

\paragraph{Function-value suboptimality ($m=1$).} 
Suppose that $m=1$ and
\begin{align}\label{eq:convergence_function_value_suboptimality}
    \p{P,p,T,t} = \p{0,0,0,1}.
\end{align}
Then~\eqref{eq:R_lower_bound_converges} implies that the {\emph{function-value suboptimality}} $\set{f_1\p{y^{(1)}_{k}} - f_1\p{y_{\star}}}_{k\in\naturals}$ converges to zero, since 
\begin{align*}
    \quadform{T}{(\bx_k-\bx_\star,\bu_k,\bu_\star)} + t^{\top}(\bFcn_k-\bFcn_{\star}) = f_1\p{y^{(1)}_{k}} - f_1\p{y_{\star}}\geq 0.
\end{align*}
Moreover,~\eqref{eq:R_lower_bound_summable} and Jensen's inequality imply that the {\emph{ergodic function-value suboptimality}} 
\begin{align*}
    \set{f_1\p{\frac{1}{k+1}\sum_{j=0}^{k}y^{(1)}_{j}} - f_1\p{y_{\star}}}_{k\in\naturals}
\end{align*}
converges to zero with rate $\mathcal{O}\p{1/k}$ since  
\begin{align*}
    f_1\p{\frac{1}{k+1}\sum_{j=0}^{k}y^{(1)}_{j}} - f_1\p{y_{\star}} \leq \frac{V(\bxi_0,\bxi_\star)}{k+1}.
\end{align*}

\paragraph{Duality gap.} 
Suppose that
\begin{align}\label{eq:convergence_duality_gap}
    \p{P,p,T,t} = \p{0,0,
    \begin{bmatrix}
        C & D & -D\\
        0 & 0 & I
    \end{bmatrix}^{\top}
    \begin{bmatrix}
        0 & -\tfrac{1}{2}I\\
        -\tfrac{1}{2}I & 0
    \end{bmatrix}
    \begin{bmatrix}
        C & D & -D\\
        0 & 0 & I
    \end{bmatrix}
    ,\mathbf{1}}.
\end{align}
Then
\begin{align} \notag
    \quadform{T}{(\bx_{k}-\bx_\star,\bu_{k},\bu_\star)}+t^{\top}(\bFcn_{k}-\bFcn_\star)
    &=\sum_{i=1}^{m}\p{ f_i\p{y^{(i)}_{k}}-f_i\p{y_\star^{(i)}}-\inner{ u_\star^{(i)}}{y^{(i)}_{k}-y^{(i)}_\star}} \\ \label{eq:duality_gap}
    &=\sum_{i=1}^{m}\p{ f_i\p{y^{(i)}_{k}}-f_i\p{y_\star}-\inner{ u_\star^{(i)}}{y^{(i)}_{k}}}\geq 0,
\end{align}
since $\sum_{i=1}^{m}u_\star^{(i)}=0$ and $y_\star^{(1)}=\ldots=y_\star^{(m)}=y_\star$ (all fixed points are fixed-point encodings).
The quantity in~\eqref{eq:duality_gap} is known as the \emph{duality gap}. Note that if $m=1$, the duality gap reduces to function-value suboptimality. The duality gap is in fact a natural generalization to the function-value suboptimality, which we motivate next (see also, e.g., \cite{Chambolle_Pock_ergodic_2016}, \cite[Theorem 3.9]{Banert_Ringh_2020}, and \cite[Section 3.1]{Banert_Bot_Csetnek_2021}). Problem~\eqref{eq:the_problem} can equivalently be written as
\begin{equation*}
    \begin{aligned}
        & \underset{\p{y^{(1)},\ldots,y^{(m)}}\in\calH^{m}}{\text{minimize}} & & \sum_{i=1}^{m}f_{i}\p{y^{(i)}}\\
        &~~\text{subject to} & & y^{(i)}=y^{(m)} {\text{ for each }} i\in\llbracket 1,m\rrbracket.
    \end{aligned}
\end{equation*}
It has the \emph{Lagrangian function} $\mathcal{L}:\calH^m\times\calH^m\to\reals$ given by
\begin{align}\label{eq:Lagrangian}
    \mathcal{L}(\by,\bu) = \sum_{i=1}^m f_i(y^{(i)})+\sum_{i=1}^{m} \inner{u^{(i)}}{y^{(m)}-y^{(i)}},
\end{align}
where $\by=\p{y^{(1)},\ldots,y^{(m)}}\in\calH^m$, and $\bu=\p{u^{(1)},\ldots,u^{(m)}}\in\calH^m$ are the dual variables. The Lagrangian function satisfies
\begin{align*}
    \mathcal{L}(\by_\star,\bu)\leq \underbrace{\mathcal{L}(\by_\star,\bu_\star)}_{= \sum_{i=1}^mf_i(y_\star)} \leq\mathcal{L}(\by,\bu_\star)
\end{align*}
for each $\by,\bu\in\calH^{m}$. In particular, $\mathcal{L}(\by_k,\bu_\star)-\mathcal{L}(\by_\star,\bu_k)$ is equal to~\eqref{eq:duality_gap} and~\eqref{eq:R_lower_bound_converges} implies that the duality gap $\set{\mathcal{L}(\by_k,\bu_\star)-\mathcal{L}(\by_\star,\bu_k)}_{k\in\naturals}$ converges to zero. Moreover,~\eqref{eq:R_lower_bound_summable} and Jensen's inequality imply that the {\emph{ergodic duality gap}} 
\begin{align*}
    \set{\mathcal{L}\p{\frac{1}{k+1}\sum_{j=0}^{k}\by_j,\bu_\star}-\mathcal{L}\p{\by_\star,\frac{1}{k+1}\sum_{j=0}^{k}\bu_j}}_{k\in\naturals}
\end{align*}
converges to zero with rate $\mathcal{O}\p{1/k}$.


\section{Main result}\label{sec:main}
This section provides a necessary and sufficient condition, in terms of the feasibility of a semidefinite program, for the existence of a quadratic Lyapunov inequality in the sense of Definition~\ref{def:Lyapunov}. First, we introduce some necessary notation. Recall $\SumToZeroMat\in\reals^{m\times(m-1)}$ defined in~\eqref{eq:SumToZeroMat_def} when $m>1$. For all the matrices defined below, the interpretation is that the block column containing $\SumToZeroMat$ is removed when $m=1$. Let
\begin{align}
    \nonumber
    E_{\text{\o},+} &= 
    \begin{bmatrix} 
        C(I-A)&D-CB&-D&CB\SumToZeroMat\\
        0 & I & 0 & 0\\
        0 & 0 & I & 0
    \end{bmatrix},
    &
    E_{\text{\o},\star} &= 
    \begin{bmatrix} 
        C &D&0&-D\SumToZeroMat\\
        0 & I & 0 & 0\\
        0 & 0 & 0 & \SumToZeroMat
    \end{bmatrix}, 
    \\ 
    E_{+,\text{\o}} &= 
    \begin{bmatrix} 
        C(A-I)&CB-D&D&-CB\SumToZeroMat\\
        0 & 0 & I &0\\
        0 & I & 0 &0
    \end{bmatrix}, &
    E_{\star,\text{\o}} &= 
    \begin{bmatrix} 
        -C &-D&0&D\SumToZeroMat\\
        0 & 0 & 0 & \SumToZeroMat\\
        0 & I & 0 & 0
    \end{bmatrix}, \label{eq:E_matrices} 
    \\ \nonumber
    E_{+,\star} &= 
    \begin{bmatrix} 
        CA &CB&D&-D\SumToZeroMat-CB\SumToZeroMat\\
        0 & 0 & I & 0\\
        0 & 0 & 0 & \SumToZeroMat
    \end{bmatrix}, &
    E_{\star,+} &= 
    \begin{bmatrix} 
        -CA & -CB&-D&D\SumToZeroMat+CB\SumToZeroMat\\
        0 & 0 & 0 & \SumToZeroMat\\
        0 & 0 & I & 0
    \end{bmatrix},
\end{align}
where $E_{i,j}\in\reals^{3m\times(n + 3m-1)}$ for each distinct $i,j\in\{\text{\o},+,\star\}$, and
\begin{equation}
\begin{aligned}
    H_{\text{\o},+}&=\begin{bmatrix}I&-I\end{bmatrix},&
    H_{+,\text{\o}}&=\begin{bmatrix}-I&I\end{bmatrix},&
    H_{\text{\o},\star}&=\begin{bmatrix}I&0\end{bmatrix},\\
    H_{\star,\text{\o}}&=\begin{bmatrix}-I&0\end{bmatrix},&
    H_{+,\star}&=\begin{bmatrix}0&I\end{bmatrix},&
    H_{\star,+}&=\begin{bmatrix}0&-I\end{bmatrix},
\end{aligned}
\label{eq:H_matrices}
\end{equation}
where $H_{i,j}\in\reals^{m\times 2m}$ for each distinct $i,j\in\{\text{\o},+,\star\}$. Define
\begin{align}
    \bMlij &= E_{i,j}^{\top}\bM_lE_{i,j}\in\sym^{n + 3m -1} \quad\text{ and }\quad \balij=H_{i,j}^{\top}\ba_l\in\reals^{2m}
    \label{eq:Mlij_alij}
\end{align}
for each distinct $i,j\in\{\text{\o},+,\star\}$ and $l\in\llbracket1,m\rrbracket$, where the $\bM_l$'s and $\ba_{l}$'s are defined in~\eqref{eq:M_l} and~\eqref{eq:a_l}, respectively. Moreover, let
\begin{equation}
\begin{aligned}
    \Sigma_{\text{\o}} &=
    \begin{bmatrix}
        I & 0 & 0 & 0\\
        0 & I & 0 & 0\\
        0 & 0 & 0 & \SumToZeroMat\\
    \end{bmatrix},
    &
    \Sigma_{+}&=
    \begin{bmatrix}
        A & B & 0 & -B\SumToZeroMat\\
        0 & 0 & I & 0\\
        0 & 0 & 0 & \SumToZeroMat\\
    \end{bmatrix},
\end{aligned}
\label{eq:Sigma_matrices}
\end{equation}
where $\Sigma_i\in\reals^{(n+2m)\times(n+3m-1)}$ for each $i\in\{\text{\o},+\}$.

\begin{theorem}[\normalfont Main result] \label{thm:Lyapunov_iff}
    Assume that Assumption~\ref{ass:ABCD_solution_and_fixed_point_condition} and Assumption~\ref{ass:well-posedness} hold, let $\rho\in[0,1]$, and suppose that $P,T\in\sym^{n+2m}$ and $p,t\in\reals^{m}$ are such that 
    \begin{align*}
        \quadform{P}{(\bx-\bx_{\star},\bu,\bu_{\star})}+p^{\top}(\bFcn-\bFcn_{\star})\geq 0 
        \quad \text{ and } \quad
        \quadform{T}{(\bx-\bx_{\star},\bu,\bu_{\star})}+t^{\top}(\bFcn-\bFcn_{\star}) \geq 0
    \end{align*}
    for each $\bxi\in \mathcal{S}$ that is algorithm-consistent for $\bfcn$, each $\bxi_{\star}\in \mathcal{S}$ that satisfies~\eqref{eq:fixed_point}, and each $\bfcn\in \mathbfcal{F}_{\bm \sigma, \bm \beta}$. Then a sufficient condition for there to exist a Lyapunov function $V:\mathcal{S}\times\mathcal{S}\to \reals$ as in~\eqref{eq:Lyapunov_form} and a residual function $R:\mathcal{S}\times\mathcal{S}\to \reals$ as in~\eqref{eq:Residual_form} such that they satisfy the $\p{P,p,T,t,\rho}$-quadratic Lyapunov inequality for algorithm~\eqref{eq:linear_system_with_nonlinearity} over the class $ \mathbfcal{F}_{\bm \sigma, \bm \beta}$ is that the following system of constraints 
    \begingroup
    \allowdisplaybreaks
    \begin{subequations}\label{eq:Lyapunov_condition}
        \begin{align}
            \label{eq:Lyapunov_condition:Lyapunov_inequality}
            \rm{C}\ref{item:Lyapunov_inequality} & 
            \left\{
            \begin{aligned}
                &\lambda_{(l,i,j)}^{\rm{C}\ref{item:Lyapunov_inequality}}\geq 0\text{ for each } l\in\llbracket 1,m\rrbracket \text{ and distinct }i,j\in\set{\text{\o},+,\star},\\ 
                &\Sigma_{\text{\o}}^{\top}\p{\rho Q - S}\Sigma_{\text{\o}} - \Sigma_{+}^{\top}Q\Sigma_{+} + \sum_{l=1}^m\sum_{\substack{i,j\in\{\text{\o},+,\star\}\\ i\neq j}}\lambda_{(l,i,j)}^{\rm{C}\ref{item:Lyapunov_inequality}}\bMlij\succeq 0,\\ 
                &\begin{bmatrix}\rho q- s\\-q\end{bmatrix} + \sum_{l=1}^m\sum_{\substack{i,j\in\{\text{\o},+,\star\}\\ i\neq j}}\lambda_{(l,i,j)}^{\rm{C}\ref{item:Lyapunov_inequality}}\balij=0,
            \end{aligned} \right.\\
            \label{eq:Lyapunov_condition:V_lower_bound}
            \rm{C}\ref{item:V_lower_bound} & 
            \left\{
            \begin{aligned}
                &\lambda_{(l,i,j)}^{\rm{C}\ref{item:V_lower_bound}}\geq 0 \text{ for each } l\in\llbracket 1,m\rrbracket \text{ and distinct }i,j\in\set{\text{\o},\star},\\
                &\Sigma_{\text{\o}}^{\top}(Q-P)\Sigma_{\text{\o}} + \sum_{l=1}^m\sum_{\substack{i,j\in\{\text{\o},\star\} \\ i\neq j}}\lambda_{(l,i,j)}^{\rm{C}\ref{item:V_lower_bound}}\bMlij\succeq 0,\\ 
                &\begin{bmatrix}q-p\\0\end{bmatrix} + \sum_{l=1}^m\sum_{\substack{i,j\in\{\text{\o},\star\}\\ i\neq j}}\lambda_{(l,i,j)}^{\rm{C}\ref{item:V_lower_bound}}\balij=0,
            \end{aligned} \right. \\
            \label{eq:Lyapunov_condition:R_lower_bound}
            \rm{C}\ref{item:R_lower_bound} & 
            \left\{
            \begin{aligned}
                &\lambda_{(l,i,j)}^{\rm{C}\ref{item:R_lower_bound}}\geq 0 \text{ for each } l\in\llbracket 1,m\rrbracket \text{ and distinct }i,j\in\set{\text{\o},\star},\\
                &\Sigma_{\text{\o}}^{\top}(S-T)\Sigma_{\text{\o}} + \sum_{l=1}^m\sum_{\substack{i,j\in\{\text{\o},\star\} \\ i\neq j}}\lambda_{(l,i,j)}^{\rm{C}\ref{item:R_lower_bound}}\bMlij \succeq 0,\\ 
                &\begin{bmatrix}s-t\\0\end{bmatrix} + \sum_{l=1}^m\sum_{\substack{i,j\in\{\text{\o},\star\}\\ i\neq j}}\lambda_{(l,i,j)}^{\rm{C}\ref{item:R_lower_bound}}\balij=0,
            \end{aligned} \right. \\
            \label{eq:Lyapunov_condition:Q_S}
            &Q,S\in\sym^{n+2m}, \\
            \label{eq:Lyapunov_condition:q_s}
            &q,s\in\reals^{m},
        \end{align}
    \end{subequations}
    \endgroup
    is feasible for the scalars $\lambda_{(l,i,j)}^{\rm{C}\ref{item:Lyapunov_inequality}}$, $\lambda_{(l,i,j)}^{\rm{C}\ref{item:V_lower_bound}}$,$\lambda_{(l,i,j)}^{\rm{C}\ref{item:R_lower_bound}}$, matrices $Q$ and $S$, and vectors $q$ and $s$.
    Moreover, if $\dim(\calH)\geq n+3m-1$ and there exists $G\in\sym_{++}^{n+3m-1}$ and $\bchi\in\reals^{2m}$ such that 
    \begin{equation}\label{eq:slater_in_theorem}
    \begin{aligned}
        &\balij^{\top}\bm{\chi} + \trace\p{\bMlij G}\leq 0 \text{ for each } l\in\llbracket 1,m\rrbracket \text{ and distinct }i,j\in\set{\text{\o},+,\star},
        \end{aligned}
    \end{equation} 
    then the feasibility of~\eqref{eq:Lyapunov_condition} is also a necessary condition.
\end{theorem}

The proof of Theorem~\ref{thm:Lyapunov_iff} is based on, for C\ref{item:Lyapunov_inequality}, C\ref{item:V_lower_bound} and C\ref{item:R_lower_bound} in Definition~\ref{def:Lyapunov}, finding the relevant conditions, respectively, and then combining these conditions together to give~\eqref{eq:Lyapunov_condition};~\eqref{eq:Lyapunov_condition:Lyapunov_inequality} correspond to C\ref{item:Lyapunov_inequality},~\eqref{eq:Lyapunov_condition:V_lower_bound} correspond to C\ref{item:V_lower_bound} and~\eqref{eq:Lyapunov_condition:R_lower_bound} correspond to C\ref{item:R_lower_bound}.

Finding the conditions for C\ref{item:Lyapunov_inequality}, C\ref{item:V_lower_bound} and C\ref{item:R_lower_bound} is done in the spirit of PEP by finding the worst-case behavior over algorithm-consistent points, their successors, fixed points~\eqref{eq:fixed_point}, and mappings in the class $\mathbfcal{F}_{\bm \sigma, \bm \beta}$. In particular, given some objective $\Phi:\mathcal{S}^3\to\reals$, the performance estimation problem we consider is
\begin{equation}\label{eq:primalPEP}\tag{PEP}
\begin{aligned}
        & \underset{}{\text{maximize}} & & \Phi(\bxi,\bxi_{+},\bxi_{\star}) \\ \nonumber
        & \text{subject to} & & \bx_{+} = (A\kron\Id) \bx + (B\kron\Id) \bu, \\ \nonumber
        & & & \by = (C\kron\Id) \bx + (D\kron\Id) \bu, \\ \nonumber
        & & & \bu \in\bm{\partial}\bfcn(\by),  \\ \nonumber
        & & & \bFcn=\bfcn(\by),   \\ \nonumber
        & & & \by_{+} = (C\kron\Id) \bx_{+} + (D\kron\Id) \bu_{+}, \\ \nonumber
        & & & \bu_{+} \in\bm{\partial}\bfcn(\by_{+}),  \\ \nonumber
        & & & \bFcn_{+}=\bfcn(\by_{+}),   \\ \nonumber
        & & & \bx_{\star} = (A\kron\Id) \bx_{\star} + (B\kron\Id) \bu_{\star}, \\ \nonumber
        & & & \by_{\star} = (C\kron\Id) \bx_{\star} + (D\kron\Id) \bu_{\star}, \\ \nonumber
        & & &  \bu_{\star} \in \bm{\partial}\bfcn(\by_{\star}), \\ \nonumber
        & & & \bFcn_{\star}=\bfcn(\by_{\star}), \\ \nonumber
        & & & \bxi=(\bx,\bu,\by,\bFcn), \\ \nonumber
        & & & \bxi_{+}=(\bx_{+},\bu_{+},\by_{+},\bFcn_{+}), \\ \nonumber
        & & & \bxi_{\star}=(\bx_{\star},\bu_{\star},\by_{\star},\bFcn_{\star}),\\ \nonumber
        & & & \bfcn \in \mathbfcal{F}_{\bm \sigma, \bm \beta},
\end{aligned}
\end{equation}
where we maximize over all variables except $A,B,C,D$, $ \mathbfcal{F}_{\bm \sigma, \bm \beta}$ and $\Phi$. Let $S_{\Phi}^\star$ be the optimal value of~\eqref{eq:primalPEP}.

We consider objective functions $\Phi$ in~\eqref{eq:primalPEP} of the form
\begin{align}\label{eq:worst_case_objective_quadratic} 
        \Phi(\bxi,\bxi_{+},\bxi_{\star})=&\quadform{Q_{\text{\o}}}{(\bx-\bx_{\star},\bu,\bu_{\star})}+q_{\text{\o}}^{\top}(\bFcn-\bFcn_{\star})+\quadform{Q_{+}}{(\bx_{+}-\bx_{\star},\bu_{+},\bu_{\star})}+q_{+}^{\top}(\bFcn_{+}-\bFcn_{\star}),
\end{align}
parameterized by $Q_{\text{\o}},Q_{+}\in\sym^{n+2m}$ and $q_{\text{\o}},q_{+}\in\reals^{m}$. For each $cond \in \set{\rm{C}\ref{item:Lyapunov_inequality},\rm{C}\ref{item:V_lower_bound},\rm{C}\ref{item:R_lower_bound}}$ separately, the parameters $Q_{\text{\o}}$, $Q_{+}$, $q_{\text{\o}}$ and $q_{+}$ are chosen such that $S_{\Phi}^{\star}\leq0$ is a necessary and sufficient condition for $cond$ to hold.

Before we proceed, we reformulate~\eqref{eq:primalPEP}, and in order to do so we introduce some helpful notation. We let
\begin{equation}
    \begin{aligned}
    \bQ&=\Sigma_{\text{\o}}^{\top}Q_{\text{\o}}\Sigma_{\text{\o}}+ \Sigma_{+}^{\top}Q_{+}\Sigma_{+}\in\sym^{n+3m-1}, &
    \bq&=(q_{\text{\o}},q_{+})\in\reals^{2m},
    \end{aligned}
    \label{eq:bQ_and_bq}
\end{equation}
where $Q_{\text{\o}}$, $Q_{+}$, $q_{\text{\o}}$, and $q_{+}$ are the parameters in the objective function $\Phi$ given in~\eqref{eq:worst_case_objective_quadratic}, and $\Sigma_{\text{\o}}$ and $\Sigma_{+}$ are given in~\eqref{eq:Sigma_matrices}. 

\begin{lemma} \label{lem:PSD_on_alg}
    Let $\mathcal{I} = \set{\text{\o},\star}$ or $\mathcal{I} = \set{\text{\o},+,\star}$, $S_{\Phi}^\star$ the optimal value of~\eqref{eq:primalPEP}, and assume that Assumption~\ref{ass:ABCD_solution_and_fixed_point_condition} and Assumption~\ref{ass:well-posedness} hold. Suppose that $\Phi$ is of the form~\eqref{eq:worst_case_objective_quadratic} and that the right-hand side of~\eqref{eq:worst_case_objective_quadratic} only depends on variables with indices in the set $\mathcal{I}$ (a variable without a subscript is interpreted to have index \o). A sufficient condition for $S_{\Phi}^\star\leq 0$ is that  the following system 
    \begin{equation}
        \left\{
        \begin{aligned}
            & \lambda_{(l,i,j)}\geq 0 \text{ for each } l\in\llbracket 1,m\rrbracket \text{ and distinct }i,j\in\mathcal{I},\\
            & - \bQ + \sum_{l=1}^{m}\sum_{\substack{i, j\in\mathcal{I}\\ i\neq j}}\lambda_{(l,i,j)}\bMlij\succeq 0,\\
            & - \bq + \sum_{l=1}^{m}\sum_{\substack{i, j\in\mathcal{I}\\ i\neq j}}\lambda_{(l,i,j)}\balij=0,
        \end{aligned}
        \right.
        \label{eq:feasibility_certificate}
    \end{equation}
    is feasible for the scalars $\lambda_{(l,i,j)}$. 
    Furthermore, if $\dim(\calH)\geq n+3m-1$, and there exists $G\in\sym_{++}^{n+3m-1}$ and $\bchi\in\reals^{2m}$ such that 
    \begin{equation}\label{eq:slater_in_lemma}
    \begin{aligned}
        &\balij^{\top}\bm{\chi} + \trace\p{\bMlij G}\leq 0 \text{ for each } l\in\llbracket 1,m\rrbracket \text{ and distinct }i,j\in\mathcal{I},
        \end{aligned}
    \end{equation}
    then~\eqref{eq:feasibility_certificate} is a necessary condition.
\end{lemma}

\begin{proof}[Proof sketch]
    The full proof is provided in Section~\ref{sec:lemma_1_proof}. The proof first reformulates~\eqref{eq:primalPEP} as a semidefinite program, forms the dual problem, which is equal to the feasibility problem~\eqref{eq:feasibility_certificate}, and shows strong duality when $\dim(\calH)\geq n+3m-1$ and~\eqref{eq:slater_in_lemma} holds.
\end{proof}

\begin{proof}[Proof of Theorem~\ref{thm:Lyapunov_iff}]

First, suppose that the parameters $(Q,q,T,t)$ are fixed in some Lyapunov function $V:\mathcal{S}\times\mathcal{S}\to \reals$ as in~\eqref{eq:Lyapunov_form} and some residual function $R:\mathcal{S}\times\mathcal{S}\to \reals$ as in~\eqref{eq:Residual_form}. We consider when $V$ and $R$ satisfy the $\p{P,p,T,t,\rho}$-quadratic Lyapunov inequality.

C\ref{item:Lyapunov_inequality} holds if $S_{\Phi}^{\star}\leq 0$ for the choice $Q_{\text{\o}}=S-\rho Q$, $q_{\text{\o}}=s-\rho q$, $Q_{+}=Q$, and $q_{+}=q$, which in turn holds if~\eqref{eq:Lyapunov_condition:Lyapunov_inequality} is feasible, according to Lemma~\ref{lem:PSD_on_alg}. 

C\ref{item:V_lower_bound} holds if $S_{\Phi}^{\star}\leq 0$ for the choice $Q_{\text{\o}}=P-Q$, $q_{\text{\o}}=p-q$, $Q_{+}=0$ and $q_{+}=0$, which in turn holds if~\eqref{eq:Lyapunov_condition:V_lower_bound} is feasible, according to Lemma~\ref{lem:PSD_on_alg}.

C\ref{item:R_lower_bound} holds if $S_{\Phi}^{\star}\leq 0$ for the choice $Q_{\text{\o}}=T-S$, $q_{\text{\o}}=t-s$, $Q_{+}=0$ and $q_{+}=0$, which in turn holds if~\eqref{eq:Lyapunov_condition:R_lower_bound} is feasible, according to Lemma~\ref{lem:PSD_on_alg}.

If in addition $\dim(\calH)\geq n+3m-1$ and~\eqref{eq:slater_in_theorem} hold, then Lemma~\ref{lem:PSD_on_alg} gives that feasibility of~\eqref{eq:Lyapunov_condition:Lyapunov_inequality}-\eqref{eq:Lyapunov_condition:R_lower_bound} is a necessary condition for C\ref{item:Lyapunov_inequality}, C\ref{item:V_lower_bound} and C\ref{item:R_lower_bound} to hold simultaneously. 

Second, note that the proof is complete if we let the parameters $(Q,q,T,t)$ free, as in~\eqref{eq:Lyapunov_condition:Q_S}-\eqref{eq:Lyapunov_condition:q_s}.
\end{proof}  

\section{Numerical examples}\label{sec:numerical_examples}
The necessary and sufficient condition \eqref{eq:Lyapunov_condition} in Theorem~\ref{thm:Lyapunov_iff} for the existence of a Lyapunov inequality is a semidefinite program of size $n+2m$ (which is below ten for all examples in Section~\ref{sec:alg_examples}) and is readily solved by standard solvers. We apply Theorem~\ref{thm:Lyapunov_iff} to each example in Section~\ref{sec:alg_examples} in two different ways:
\begin{enumerate}[B1.]
    \item We find the smallest possible $\rho\in[0,1[$, via bisection search, such that a $\p{P,p,T,t,\rho}$-Lyapunov inequality exists, where $\p{P,p,T,t}$ is chosen as in~\eqref{eq:linear_rate_distance_to_the_solution}, which implies that the squared distance to the solution convergence $\rho$-linearly to zero. The tolerance for the bisection search is set to $0.001$ and $i$ is set to $1$ in~\eqref{eq:linear_rate_distance_to_the_solution} for all examples. 
    \label{item:numeric_linear}
    \item We fix $\rho = 1$ and find a range of algorithm parameters for which there exists a $\p{P,p,T,t,\rho}$-Lyapunov inequality, where $\p{P,p,T,t}$ is chosen as in~\eqref{eq:convergence_function_value_suboptimality} if $m=1$ and~\eqref{eq:convergence_duality_gap} if $m>1$, implying ($\mathcal{O}\p{1/k}$ ergodic) convergence of the function-value suboptimality and duality gap, respectively. The parameter range is evaluated on a square grid of size $0.01\times0.01$.
    \label{item:numeric_ergodic}
\end{enumerate}

\subsection{Douglas--Rachford method}
Consider the Douglas--Rachford method in Section~\ref{sec:douglas_rachford} in the case when $f_1\in\mathcal{F}_{1,2}$, $f_2\in\mathcal{F}_{0,\infty}$ and $\lambda = 1$. Figure~\ref{fig:Douglas–Rachford_linear_rates} shows the $\rho$ we obtain via~B\ref{item:numeric_linear}. In particular, note that we recover the already known tight rates given in \cite[Theorem~2]{Giselsson2017Linear}.

\begin{figure}[t!]
    \centering
    \pgfplotstableread[col sep=comma,]{DR_linear_sigmaEq1_betaEq2_lambdaEq1.tex}\datatable
    \begin{tikzpicture}
        \begin{axis}[
            width=0.75\textwidth,
            height=5cm,
            legend style={at={(0.985,0.05)},anchor=south east,nodes=right,cells={line width=1pt}},
            ylabel={$\rho$},
            xlabel={$\gamma$},
            xmin=0,
            xmax=5,
            ymin=0,
            ytick distance=0.2,
            clip marker paths=true,
            ylabel style={rotate=-90},
            grid=both]
            
            \addplot [color=matlabcolor1op,only marks,on layer=background] table [x={x-axis}, y={a}]{\datatable};
            \addlegendentry{Thm.~\ref{thm:Lyapunov_iff}}
            
            \addplot[color=matlabcolor4!50!matlabcolor1op,line width=1.5pt] table [x={x-axis}, y={b}]{\datatable};
            \addlegendentry{\cite[Thm.~2]{Giselsson2017Linear}}
        \end{axis}
    \end{tikzpicture}
    \caption{B\ref{item:numeric_linear} applied to the Douglas–Rachford method (see Section~\ref{sec:douglas_rachford}) when $f_1\in\mathcal{F}_{1,2}$, $f_2\in\mathcal{F}_{0,\infty}$ and $\lambda = 1$, and the tight convergence rate given in \cite[Theorem~2]{Giselsson2017Linear}.}
    \label{fig:Douglas–Rachford_linear_rates}
\end{figure}
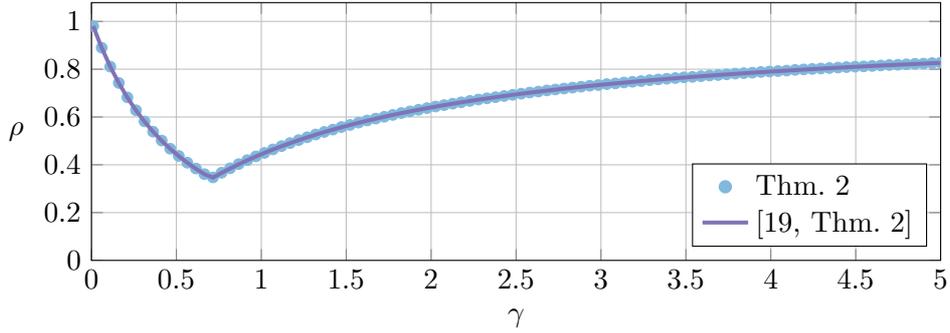

\subsection{Proximal gradient method with heavy-ball momentum}
Consider the gradient method with heavy-ball momentum in Section~\ref{sec:gradient_method_heavy-ball_momentum} and the proximal operator extension in Section~\ref{sec:proximal_gradient_method_heavy-ball_momentum}. Note that the method in Section~\ref{sec:proximal_gradient_method_heavy-ball_momentum} reduces to the one in Section~\ref{sec:gradient_method_heavy-ball_momentum} if the proximal operator is removed and either $\delta_1=0$ or $\delta_2=0$. 

Figure~\ref{fig:subfig:heavy_ergodic} contains the parameter region we obtain via~B\ref{item:numeric_ergodic} for the gradient method with heavy-ball momentum when $f_1\in\mathcal{F}_{0,1}$. Note that we improve on the parameter region given in~\cite{Ghadimi2015Global} that guarantees $\mathcal{O}\p{1/k}$ ergodic convergence of the function-value suboptimality.

Figure~\ref{fig:subfig:prox_heavy_ergodic} contains the parameter region we obtain via~B\ref{item:numeric_ergodic} for the (proximal) gradient method with heavy-ball momentum when $f_1\in\mathcal{F}_{0,1}$ (and $f_2\in\mathcal{F}_{0,\infty}$). In particular, note how the feasible parameter region is affected by adding a proximal term---having the momentum term inside the proximal evaluation ($\delta_2=0$) gives a slightly smaller region, and having it outside ($\delta_1=0$) makes it even smaller. 

Figure~\ref{fig:subfig:heavy_linear} shows the $\rho$ we obtain via~B\ref{item:numeric_linear} for the gradient method with heavy-ball momentum when $f_1\in\mathcal{F}_{1,10}$. Note that we improve on the rates given in~\cite{Ghadimi2015Global} and range of allowable momentum parameters $\delta$ that guarantee linear convergence.

\begin{figure}[t!]
    \centering
    \begin{subfigure}{0.49\textwidth}
        \pgfplotstableread[col sep=comma,]{Polyakgrad_Ghadimi.tex}\datatable
        \pgfplotstableread[col sep=comma,]{Polyakgrad_CVX_params.tex}\datatableour
        \begin{tikzpicture}
            \begin{axis}[
                width=\textwidth,
                height=5cm,
                legend style={at={(0.02,0.98)},anchor=north west,nodes=right,cells={line width=2pt}},
                ylabel={$\gamma$},
                xlabel={$\delta$},
                xmin=-1,
                xmax=1,
                ymin=0,
                ylabel style={rotate=-90},
                ytick distance=0.5]
                \addplot [matlabcolor1op,fill=matlabcolor1op] table [x={x-axis}, y={a}]{\datatableour}\closedcycle;
                \addlegendentry{Thm.~\ref{thm:Lyapunov_iff}}
                \addplot [matlabcolor4op,fill=matlabcolor4op] table [x={x-axis}, y={a}]{\datatable}\closedcycle;
                \addlegendentry{\cite[Thm.~1]{Ghadimi2015Global}}
            \end{axis}
        \end{tikzpicture}
        \caption{B\ref{item:numeric_ergodic} applied to the gradient method with heavy-ball momentum  when  $f_1\in\mathcal{F}_{0,1}$, and compared to the range given in \cite[Theorem~1]{Ghadimi2015Global} that gives $\mathcal{O}\p{1/k}$ ergodic convergence of the function-value suboptimality.}
        \label{fig:subfig:heavy_ergodic}
    \end{subfigure}
    \hfill
    \begin{subfigure}{0.49\textwidth}
        \pgfplotstableread[col sep=comma,]{Polyakgrad_CVX_params.tex}\datatable
        \pgfplotstableread[col sep=comma,]{Polyakproxgrad_CVX_params_sweep_delta2_delta1_0.tex}\datatablesweepdeltatwo
        \pgfplotstableread[col sep=comma,]{Polyakproxgrad_CVX_params_sweep_delta1_delta2_0.tex}\datatablesweepdeltaone
        \begin{tikzpicture}
            \begin{axis}[
                width=\textwidth,
                height=5cm,
                legend style={at={(0.02,0.98)},anchor=north west,nodes=right,cells={line width=2pt}},
                ylabel={$\gamma$},
                xlabel={$\delta$, $\delta_1$, or $\delta_2$},
                xmin=-1,
                xmax=1,
                ymin=0,
                ylabel style={rotate=-90},
                ytick distance=0.5]
                \addplot [matlabcolor1op,fill=matlabcolor1op] table [x={x-axis}, y={a}]{\datatable}\closedcycle;
                \addlegendentry{Alg. $§$\ref{sec:gradient_method_heavy-ball_momentum}}
                \addplot [matlabcolor4!50!matlabcolor1op,fill=matlabcolor4!50!matlabcolor1op] table [x={x-axis}, y={a}]{\datatablesweepdeltaone}\closedcycle;
                \addlegendentry{Alg. $§$\ref{sec:proximal_gradient_method_heavy-ball_momentum}, $\delta_2=0$}
                \addplot [matlabcolor7op,fill=matlabcolor7op] table [x={x-axis}, y={a}]{\datatablesweepdeltatwo}\closedcycle;
                \addlegendentry{Alg. $§$\ref{sec:proximal_gradient_method_heavy-ball_momentum}, $\delta_1=0$}
            \end{axis}
        \end{tikzpicture}
        \caption{B\ref{item:numeric_ergodic} applied to the (proximal) gradient method with heavy-ball momentum when $f_1\in\mathcal{F}_{0,1}$ (and $f_2\in\mathcal{F}_{0,\infty}$).\\ \textit{ } \\ }
        \label{fig:subfig:prox_heavy_ergodic}
    \end{subfigure}

    \begin{subfigure}{0.75\textwidth}
        \pgfplotstableread[col sep=comma,]{Heavy-ball_linear_Ghadimi.tex}\datatableGhadimi
        \pgfplotstableread[col sep=comma,]{Heavy-ball_linear_PEP.tex}\datatablePEP
        \begin{tikzpicture}
          \begin{axis}[
            width=\textwidth,
            height=5cm,
            legend style={at={(0.985,0.05)},anchor=south east,nodes=right,cells={line width=1pt}},
            ylabel={$\rho$},
            xlabel={$\delta$},
            xmin=-0.6,
            xmax=1.0,
            ymin=0,
            clip marker paths=true,
            xtick distance=0.2,
            ytick distance=0.2,
            ylabel style={rotate=-90},
            grid=both]
    
            \addplot [color=matlabcolor1op,only marks] table [x={x-axis}, y={a}]{\datatablePEP};
            \addlegendentry{Thm.~\ref{thm:Lyapunov_iff}}

            \addplot [color=matlabcolor4op,line width=1.5pt] table [x={x-axis}, y={a}]{\datatableGhadimi};
            \addlegendentry{\cite[Thm. 2]{Ghadimi2015Global}}
            
            \end{axis}
        \end{tikzpicture}
        \caption{B\ref{item:numeric_linear} applied to the gradient method with heavy-ball momentum  when  $f_1\in\mathcal{F}_{1,10}$ and $\gamma = 1/10$.}
        \label{fig:subfig:heavy_linear}
    \end{subfigure}
    
    \caption{Convergence analysis of the (proximal) gradient method with heavy-ball momentum (see  Section~\ref{sec:gradient_method_heavy-ball_momentum} and Section~\ref{sec:proximal_gradient_method_heavy-ball_momentum}).}
    \label{fig:heavy}
\end{figure}
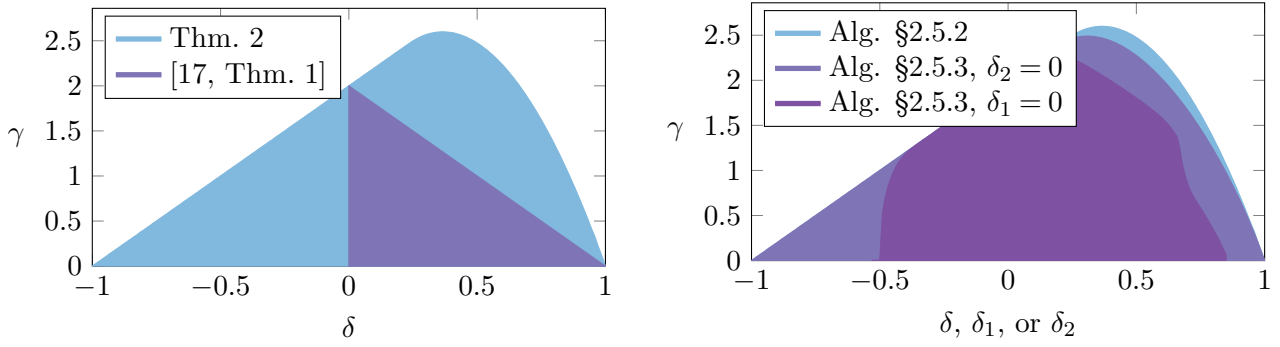
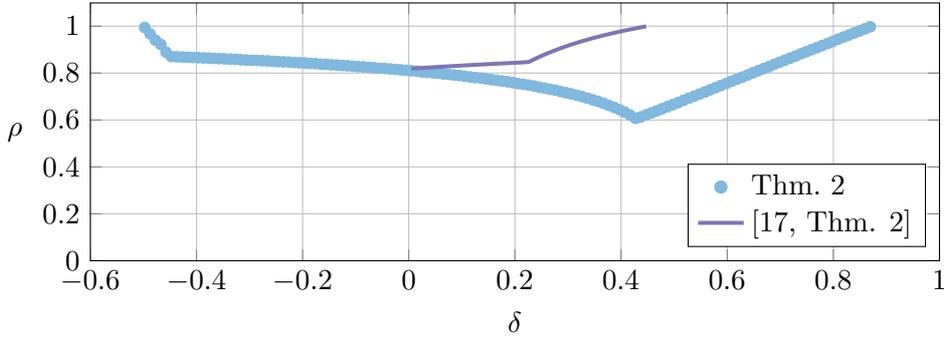

\subsection{Davis--Yin three-operator splitting method}

Consider the three-operator splitting method by Davis--Yin in Section~\ref{sec:davis_yin} in the case when $f_1\in\mathcal{F}_{0,\beta_1}$, $f_2\in\mathcal{F}_{1,2}$, $f_3\in\mathcal{F}_{0,\infty}$, $\gamma = 1/2$ and $\lambda = 1$. Figure~\ref{fig:Davis-Yin_linear_rates} shows the $\rho$ we obtain via~B\ref{item:numeric_linear}. In particular, note that we improve on the rates given in \cite[Theorem~D.6]{Davis_Yin_3op_arxiv} and \cite[Theorem~3]{pmlr-v80-pedregosa18a}.

\begin{figure}[t!]
    \centering
    \pgfplotstableread[col sep=comma,]{DY_linear_rate6.tex}\datatable
    \begin{tikzpicture}
      \begin{axis}[
        width=0.75\textwidth,
        height=5cm,
        legend style={at={(0.985,0.05)},anchor=south east,nodes=right,cells={line width=1pt}},
        ylabel={$\rho$},
        xlabel={$\beta_1$},
        xmin=0,
        xmax=40,
        ymin=0, 
        clip marker paths=true,
        xtick distance=5,
        ytick distance=0.2,
        ylabel style={rotate=-90},
        grid=both]

        \addplot [color=matlabcolor1op,only marks] table [x={x-axis}, y={a}]{\datatable};
        \addlegendentry{Thm.~\ref{thm:Lyapunov_iff}}

        \addplot[color=matlabcolor4op,line width=1.5pt] table [x={x-axis}, y={b}]{\datatable};
        \addlegendentry{\cite[Thm. D.6]{Davis_Yin_3op_arxiv}}

        \addplot[color=matlabcolor5!80!gray,line width=1.5pt] table [x={x-axis}, y={c}]{\datatable};
        \addlegendentry{\cite[Thm. 3]{pmlr-v80-pedregosa18a}}
        
        \end{axis}
    \end{tikzpicture}
    \caption{B\ref{item:numeric_linear} applied to the three-operator splitting method by Davis and Yin (see Section~\ref{sec:davis_yin}) when $f_1\in\mathcal{F}_{0,\beta_1}$, $f_2\in\mathcal{F}_{1,2}$, $f_3\in\mathcal{F}_{0,\infty}$, $\gamma = 1/2$ and $\lambda = 1$, the linear convergence rate given in \cite[Theorem~D.6]{Davis_Yin_3op_arxiv}, and the linear convergence rate given in \cite[Theorem~3]{pmlr-v80-pedregosa18a}.}
    \label{fig:Davis-Yin_linear_rates}
\end{figure}
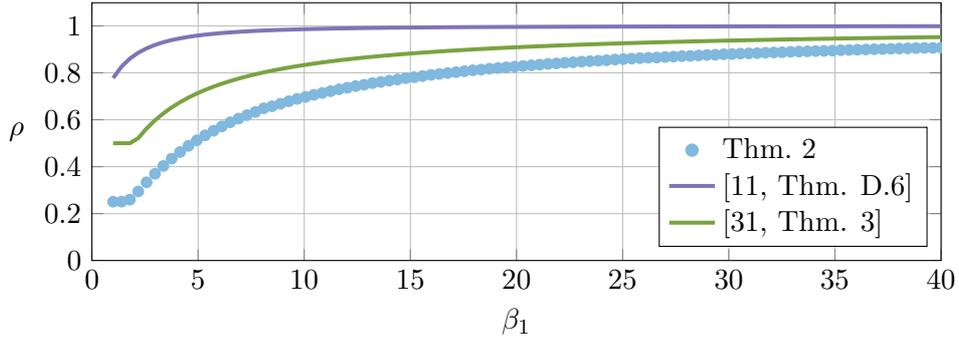

\subsection{Chambolle--Pock method}
Consider the special case of the Chambolle--Pock method when the linear operator is restricted to be the identity operator $\Id$, as presented in Section~\ref{sec:chambolle_pock}. Standard convergence proofs, e.g., the ones in \cite{Chambolle_Pock_ergodic_2011}, allow in this setting for $\theta=1$ and $\tau_1,\tau_2>0$ satisfying $\tau_1\tau_2<1$.

Figure~\ref{fig:ChambollePockParameterRegion} shows the range of parameters $\theta$, $\tau_1$, and $\tau_2$ when $\tau_1=\tau_2\geq 0.5$ that we obtain via~B\ref{item:numeric_ergodic} when $f_1,f_2\in\mathcal{F}_{0,\infty}$. This is a significantly larger region than what traditional analyses allow for. In particular, we see that $\theta\neq 1$ is a valid choice and that $\tau_1\tau_2>1$ is also valid for many choices of $\theta$. Moreover, for comparison, Figure~\ref{fig:ChambollePockParameterRegion} also contains the region if we add the additional restriction in~\eqref{eq:Lyapunov_condition} that 
\begin{align}\label{eq:ChambollePock:Qq_restriction}
    Q = 
    \begin{bmatrix}
        Q_{xx}& 0\\
        0 & 0
    \end{bmatrix},
\end{align}
 where $Q_{xx}\in\sym^{n}$ and modify $P$ in~B\ref{item:numeric_ergodic} so that 
\begin{align}\label{eq:ChambollePock:P_modification}
    P = 
    \begin{bmatrix}
        I & 0\\
        0 & 0
    \end{bmatrix},
\end{align}
where $I$ is the identity matrix of size $n\times n$. With these additional restrictions, we recover exactly the traditional convergence region.

Figure~\ref{fig:ChambollePockLinear} shows the $\rho$ that we obtain via~B\ref{item:numeric_linear} when $f_1,f_2\in\mathcal{F}_{0.05,50}$ in the region when $\tau_1=\tau_2\geq 0.5$. In particular, we note that the smallest $\rho$ is obtained for the parameters $\tau_1=\tau_2=1.6$ and $\theta=0.22$, giving a value of $\rho=0.8812$. If we restrict to the feasible parameter region in Figure~\ref{fig:ChambollePockParameterRegion}, the optimal parameters are $\tau_1=\tau_2=1.5$ and $\theta=0.35$ with $\rho=0.8891$. Both these rates are significantly better than what can be achieved with traditional parameter choices, where the optimal choice is $\tau_1=\tau_2=0.99$ and $\theta=1$ giving $\rho=0.9266$.

\begin{figure}[t!]
    \centering
    \begin{subfigure}[t]{0.49\textwidth}
        \pgfplotstableread[col sep=comma,]{ChambollePock_CVX_theta_vs_sigtau_bounds.tex}\datatable
        \pgfplotstableread[col sep=comma,]{ChambollePock_CVX_theta_vs_sigtau_bounds_restrictedLyapunov.tex}\datatablerestrictQ
        \begin{tikzpicture}[baseline,remember picture]
            \begin{axis}[
            width=\textwidth,
            height=5cm,
            legend style={at={(0.98,0.98)},anchor=north east,nodes=right,cells={line width=2pt}},
            ylabel={$\theta$},
            xlabel={$\tau_1=\tau_2$},
            xmin=0.5,
            xmax=1.75,
            ymin=-0.5,
            ymax=8.0,
            ylabel style={rotate=-90},
            ytick distance=2]
                \addplot [name path=ub,matlabcolor1op] table [x={x-axis}, y={b}]{\datatable};
                \addplot [name path=lb,matlabcolor1op,forget plot] table [x={x-axis}, y={a}]{\datatable};
                \addplot [matlabcolor1op,forget plot] fill between [of=lb and ub];
                \addlegendentry{Thm.~\ref{thm:Lyapunov_iff}}
                
                \addplot [matlabcolor7op,line width=1pt] table [x={x-axis}, y={a}]{\datatablerestrictQ};
                \addlegendentry{Thm.~\ref{thm:Lyapunov_iff}---restricted}
                
            \end{axis}
        \end{tikzpicture}
        \caption{B\ref{item:numeric_ergodic} applied to the Chambolle--Pock method in the region $\tau_1=\tau_2\geq 0.5$ when $f_1,f_2\in\mathcal{F}_{0,\infty}$, and also with the additional restrictions in~\eqref{eq:ChambollePock:Qq_restriction} and~\eqref{eq:ChambollePock:P_modification}.}
        \label{fig:ChambollePockParameterRegion}
    \end{subfigure}
    \begin{subfigure}[t]{0.49\textwidth}
        \pgfplotstableread[col sep=comma,]{ChambollePock_linear_full_region.tex}\datatable
        \begin{tikzpicture}[baseline,remember picture]
            \begin{axis}[
            width=\textwidth,
            height=5cm,
            view={0}{90},
            colorbar horizontal,
            colorbar style={ylabel={$\rho$},ylabel style={rotate=-90}},
            colormap name=mycolormap,
            ylabel={$\theta$},
            xlabel={$\tau_1=\tau_2$},
            xmin=0.5,
            xmax=1.75,
            ymin=-0.5,
            ymax=8.0,
            ylabel style={rotate=-90},
            ytick distance=2,mesh/ordering=y varies]
                \addplot3 [surf,mesh/rows=11]
                table [x={sigtau}, y={theta}, z={rho}]{\datatable};
            \end{axis}
        \end{tikzpicture}
        \caption{B\ref{item:numeric_linear} applied to the Chambolle--Pock method in the region where $\tau_1=\tau_2\geq 0.5$ when $f_1,f_2\in\mathcal{F}_{0.05,50}$.}
        \label{fig:ChambollePockLinear}
    \end{subfigure}
    \caption{Convergence analysis of the Chambolle--Pock method (see  Section~\ref{sec:chambolle_pock}).}
    \label{fig:ChambollePock}
\end{figure}
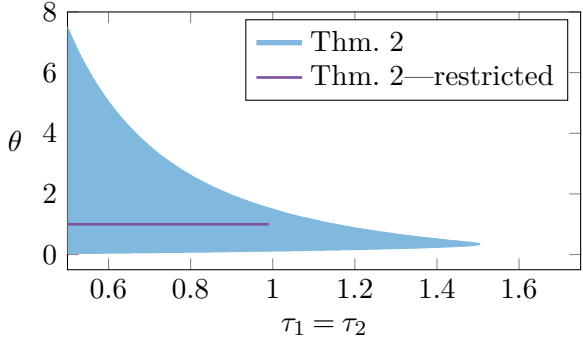
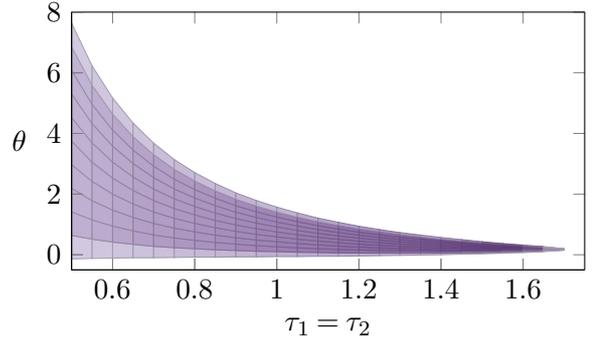  

\section{Proof of Lemma~\ref{lem:PSD_on_alg}}\label{sec:lemma_1_proof}
We prove Lemma~\ref{lem:PSD_on_alg} only in the case $\mathcal{I} = \set{\text{\o},+,\star}$, as the case $\mathcal{I} = \set{\text{\o},\star}$ is analogous. Recall that we assume that Assumption~\ref{ass:ABCD_solution_and_fixed_point_condition} and Assumption~\ref{ass:well-posedness} hold.

\begin{proof}[Proof of Lemma~\ref{lem:PSD_on_alg}]
We prove Lemma~\ref{lem:PSD_on_alg} in a sequence of steps:

\paragraph{Formulating the primal semidefinite program.} Recall that $S_{\Phi}^\star$ is the optimal value of~\eqref{eq:primalPEP}. By Corollary~\ref{cor:interpolation}, the constraints of~\eqref{eq:primalPEP} can equivalently be written as
\begin{equation}\label{eq:PEP_after_interpolation}
    \begin{aligned}
    &  \bx_{+} = (A\kron\Id) \bx + (B\kron\Id) \bu, \\ 
    & \by = (C\kron\Id) \bx + (D\kron\Id) \bu,  \\ 
    & \by_{+} = (C\kron\Id) \bx_{+} + (D\kron\Id) \bu_{+}, \\ 
    & \bx_{\star} = (A\kron\Id) \bx_{\star} + (B\kron\Id) \bu_{\star}, \\ 
    & \by_{\star} = (C\kron\Id) \bx_{\star} + (D\kron\Id) \bu_{\star}, \\ 
    & {\hbox{{\bf for each } $l\in\llbracket 1,m\rrbracket$}}\\ 
    &  \indentconstr\ba_{l}^{\top}(\bFcn-\bFcn_{+}) + \quadform{\bM_l}{(\by-\by_{+},\bu,\bu_{+})}\leq 0,\\ 
    &  \indentconstr\ba_{l}^{\top}(\bFcn_{+}-\bFcn) + \quadform{\bM_l}{(\by_{+}-\by,\bu_{+},\bu)}\leq 0,\\ 
    &  \indentconstr\ba_{l}^{\top}(\bFcn-\bFcn_{\star}) + \quadform{\bM_l}{(\by-\by_{\star},\bu,\bu_{\star})}\leq 0,\\ 
    &  \indentconstr\ba_{l}^{\top}(\bFcn_{\star}-\bFcn) + \quadform{\bM_l}{(\by_{\star}-\by,\bu_{\star},\bu)}\leq 0,\\ 
    &  \indentconstr\ba_{l}^{\top}(\bFcn_{+}-\bFcn_{\star}) + \quadform{\bM_l}{(\by_{+}-\by_{\star},\bu_{+},\bu_{\star})}\leq 0,\\ 
    &  \indentconstr\ba_{l}^{\top}(\bFcn_{\star}-\bFcn_{+}) + \quadform{\bM_l}{(\by_{\star}-\by_{+},\bu_{\star},\bu_{+})}\leq 0,\\ 
    &  \indentconstr\quadform{\bM_l}{(0,\bu,\bu)} \leq 0,\\ 
    &  \indentconstr\quadform{\bM_l}{(0,\bu_{+},\bu_{+})}\leq 0,\\ 
    &  \indentconstr\quadform{\bM_l}{(0,\bu_{\star},\bu_{\star})}\leq 0.\\ 
    & {\hbox{{\bf end}}}
    \end{aligned}
\end{equation}
By Corollary~\ref{cor:interpolation}, the last three constraints can be dropped since they encode $0\leq 0$. By inserting the $\by$, $\by_{+}$, and $\by_{\star}$ equalities and using the notation $\XId{X}=(X\kron\Id)$, the constraints in~\eqref{eq:PEP_after_interpolation} can be written as
\begin{equation}\label{eq:PEP_interpolation_constraints_1}
    \begin{aligned}
     & \bx_{+} = \XId{A} \bx + \XId{B} \bu,   \\ 
     & \bx_{\star} = \XId{A} \bx_{\star} + \XId{B} \bu_{\star}, \\ 
     & {\hbox{{\bf for each } $l\in\llbracket 1,m\rrbracket$}}\\ 
     &  \indentconstr\ba_{l}^{\top}(\bFcn-\bFcn_{+}) + \quadform{\bM_l}{(\XId{C} (\bx-\bx_{+}) + \XId{D}\bu - \XId{D}\bu_{+},\bu,\bu_{+})}\leq 0,\\ 
     &  \indentconstr\ba_{l}^{\top}(\bFcn_{+}-\bFcn) + \quadform{\bM_l}{(\XId{C} (\bx_{+}-\bx) -\XId{D}\bu + \XId{D} \bu_{+},\bu_{+},\bu)}\leq 0,\\ 
     &  \indentconstr\ba_{l}^{\top}(\bFcn-\bFcn_{\star}) + \quadform{\bM_l}{(\XId{C} (\bx-\bx_{\star}) + \XId{D}\bu- \XId{D}\bu_{\star},\bu,\bu_{\star})}\leq 0, \\ 
     &  \indentconstr\ba_{l}^{\top}(\bFcn_{\star}-\bFcn) + \quadform{\bM_l}{(-\XId{C} (\bx-\bx_{\star}) -\XId{D}\bu + \XId{D} \bu_{\star},\bu_{\star},\bu)}\leq 0,\\ 
     &  \indentconstr\ba_{l}^{\top}(\bFcn_{+}-\bFcn_{\star}) + \quadform{\bM_l}{(\XId{C} (\bx_{+}-\bx_{\star}) + \XId{D} \bu_{+} - \XId{D}\bu_{\star},\bu_{+},\bu_{\star})}\leq 0, \\ 
     &  \indentconstr\ba_{l}^{\top}(\bFcn_{\star}-\bFcn_{+}) + \quadform{\bM_l}{(-\XId{C} (\bx_{+}-\bx_{\star}) -\XId{D} \bu_{+} + \XId{D} \bu_{\star},\bu_{\star},\bu_{+})}\leq 0.\\ 
     & {\hbox{{\bf end}}}
    \end{aligned}
\end{equation}
Using the equality 
$\bx-\bx_{+} = (\bx - \bx_{\star}) - (\XId{A}(\bx - \bx_{\star}) + \XId{B}(\bu - \bu_{\star}))$ in the first two inequalities and inserting the $\bx_{+}$ and $\bx_\star$ equalities in the last two inequalities,~\eqref{eq:PEP_interpolation_constraints_1} can equivalently be written as
\begin{equation}\label{eq:PEP_interpolation_constraints_2}
    \begin{aligned}
    & \bx_{+} = \XId{A} \bx + \XId{B} \bu,   \\ 
    & \bx_{\star} = \XId{A}\bx_{\star} + \XId{B}\bu_{\star}, \\ 
    & {\hbox{{\bf for each } $l\in\llbracket 1,m\rrbracket$}}\\ 
    &  \indentconstr\ba_{l}^{\top}(\bFcn-\bFcn_{+}) + \quadform{\bM_l}{(\XId{(C(I-A))}(\bx-\bx_{\star})  + \XId{(D-CB)}\bu - \XId{D}\bu_{+} + \XId{(CB)}\bu_{\star},\bu,\bu_{+})}\leq 0,\\ 
     &  \indentconstr\ba_{l}^{\top}(\bFcn_{+}-\bFcn) + \quadform{\bM_l}{(\XId{(C(A-I))}(\bx-\bx_{\star})  + \XId{(CB-D)}\bu + \XId{D}\bu_{+}-\XId{(CB)}\bu_{\star},\bu_{+},\bu)}\leq 0,\\ 
     &  \indentconstr\ba_{l}^{\top}(\bFcn-\bFcn_{\star}) + \quadform{\bM_l}{(\XId{C} (\bx-\bx_{\star}) + \XId{D}\bu- \XId{D}\bu_{\star},\bu,\bu_{\star})}\leq 0, \\ 
     &  \indentconstr\ba_{l}^{\top}(\bFcn_{\star}-\bFcn) + \quadform{\bM_l}{(-\XId{C} (\bx-\bx_{\star}) -\XId{D}\bu + \XId{D} \bu_{\star},\bu_{\star},\bu)}\leq 0,\\ 
     &  \indentconstr\ba_{l}^{\top}(\bFcn_{+}-\bFcn_{\star}) + \quadform{\bM_l}{ \XId{(CA)}(\bx - \bx_{\star}) + \XId{(CB)}\bu + \XId{D} \bu_{+} - \XId{(D+CB)}\bu_{\star},\bu_{+},\bu_{\star})}\leq 0, \\ 
     &  \indentconstr\ba_{l}^{\top}(\bFcn_{\star}-\bFcn_{+}) + \quadform{\bM_l}{(-\XId{(CA)}(\bx - \bx_{\star}) - \XId{(CB)}\bu - \XId{D} \bu_{+} + \XId{(D+CB)}\bu_{\star},\bu_{\star},\bu_{+})}\leq 0,\\ 
     & {\hbox{{\bf end}}}
    \end{aligned}
\end{equation}
and using the same equality $\bx_{+}-\bx_\star=\XId{A}(\bx - \bx_{\star}) + \XId{B}(\bu - \bu_{\star})$, the objective function $\Phi(\bxi,\bxi_{+},\bxi_{\star})$ of~\eqref{eq:primalPEP}, given in~\eqref{eq:worst_case_objective_quadratic}, can be written as 
\begin{align}\notag
        \Phi(\bxi,\bxi_{+},\bxi_{\star})  =& \quadform{Q_{\text{\o}}}{(\bx-\bx_{\star},\bu,\bu_{\star})}+q_{\text{\o}}^{\top}(\bFcn-\bFcn_{\star}) \\  \label{eq:worst_case_objective_quadratic_inserted}
        & + \quadform{Q_{+}}{(\XId{A}(\bx-\bx_{\star}) + \XId{B}\bu - \XId{B}\bu_{\star}),\bu_{+},\bu_{\star})}+q_{+}^{\top}(\bFcn_{+}-\bFcn_{\star}).
\end{align}
Therefore, the first equality in~\eqref{eq:PEP_interpolation_constraints_2} can be dropped since nothing else in~\eqref{eq:PEP_interpolation_constraints_2} and~\eqref{eq:worst_case_objective_quadratic_inserted} depend on $\bx_+$.
Moreover, by replacing $\bx - \bx_{\star}$ with $\Delta\bx$, we get that~\eqref{eq:PEP_interpolation_constraints_2} can equivalently be written as
\begin{equation}\label{eq:PEP_interpolation_constraints_3}
    \begin{aligned}
    & \bx_{\star} = \XId{A}\bx_{\star} + \XId{B}\bu_{\star}, \\ 
    & {\hbox{{\bf for each } $l\in\llbracket 1,m\rrbracket$}}\\ 
    &  \indentconstr\ba_{l}^{\top}(\bFcn-\bFcn_{+}) + \quadform{\bM_l}{(\XId{(C(I-A))}\Delta\bx  + \XId{(D-CB)}\bu - \XId{D}\bu_{+} + \XId{(CB)}\bu_{\star},\bu,\bu_{+})}\leq 0,\\ 
     &  \indentconstr\ba_{l}^{\top}(\bFcn_{+}-\bFcn) + \quadform{\bM_l}{(\XId{(C(A-I))}\Delta\bx  + \XId{(CB-D)}\bu + \XId{D}\bu_{+}-\XId{(CB)}\bu_{\star},\bu_{+},\bu)}\leq 0,\\ 
     &  \indentconstr\ba_{l}^{\top}(\bFcn-\bFcn_{\star}) + \quadform{\bM_l}{(\XId{C} \Delta\bx + \XId{D}\bu- \XId{D}\bu_{\star},\bu,\bu_{\star})}\leq 0, \\ 
     &  \indentconstr\ba_{l}^{\top}(\bFcn_{\star}-\bFcn) + \quadform{\bM_l}{(-\XId{C} \Delta\bx -\XId{D}\bu + \XId{D} \bu_{\star},\bu_{\star},\bu)}\leq 0,\\ 
     &  \indentconstr\ba_{l}^{\top}(\bFcn_{+}-\bFcn_{\star}) + \quadform{\bM_l}{ \XId{(CA)}\Delta\bx + \XId{(CB)}\bu + \XId{D} \bu_{+} - \XId{(D+CB)}\bu_{\star},\bu_{+},\bu_{\star})}\leq 0, \\ 
     &  \indentconstr\ba_{l}^{\top}(\bFcn_{\star}-\bFcn_{+}) + \quadform{\bM_l}{(-\XId{(CA)}\Delta\bx - \XId{(CB)}\bu - \XId{D} \bu_{+} + \XId{(D+CB)}\bu_{\star},\bu_{\star},\bu_{+})}\leq 0,\\ 
     & {\hbox{{\bf end}}}
    \end{aligned}
\end{equation}
and that~\eqref{eq:worst_case_objective_quadratic_inserted} can equivalently be written as 
\begin{align}\notag
        \Phi(\bxi,\bxi_{+},\bxi_{\star})  =& \quadform{Q_{\text{\o}}}{(\Delta\bx,\bu,\bu_{\star})}+q_{\text{\o}}^{\top}(\bFcn-\bFcn_{\star}) \\ \label{eq:worst_case_objective_quadratic_inserted_translated}
        &+ \quadform{Q_{+}}{(\XId{A}\Delta\bx + \XId{B}\bu - \XId{B}\bu_{\star},\bu_{+},\bu_{\star})}+q_{+}^{\top}(\bFcn_{+}-\bFcn_{\star}).
\end{align}
The first line in~\eqref{eq:PEP_interpolation_constraints_3} and~\eqref{eq:ABCD_null_condition} in Assumption~\ref{ass:ABCD_solution_and_fixed_point_condition} imply that 
\begin{align*}
    \bu_\star =
    \begin{cases}
        0 & \text{if }m=1, \\
        \XId{\SumToZeroMat}\hat{\bu}_\star & \text{if }m>1.
    \end{cases}
\end{align*}
for some $\hat{\bu}_\star\in\calH^{m-1}$, where $\SumToZeroMat$ is defined in~\eqref{eq:SumToZeroMat_def}. 
This implies that the first line in~\eqref{eq:PEP_interpolation_constraints_3} can be written as $\bx_{\star} = \XId{A}\bx_{\star}$ if $m=1$ and $\bx_{\star} = \XId{A}\bx_{\star} + \XId{\p{B\SumToZeroMat}}\hat{\bu}_{\star}$ if $m>1$. 
Moreover, note that nothing else in~\eqref{eq:PEP_interpolation_constraints_3} and~\eqref{eq:worst_case_objective_quadratic_inserted_translated} depend on $\bx_{\star}$.
Therefore, $\bx_{\star} = 0$ is a valid choice in the $m=1$ case and in the $m>1$ case~\eqref{eq:ABCD_range_condition} in Assumption~\ref{ass:ABCD_solution_and_fixed_point_condition} gives that the first line in~\eqref{eq:PEP_interpolation_constraints_3} can be dropped since for each $\hat{\bu}_\star\in\calH^{m-1}$ there exists an $\bx_{\star}\in\calH^{n}$ such that $\bx_{\star} = \XId{A}\bx_{\star} + \XId{\p{B\SumToZeroMat}}\hat{\bu}_{\star}$ is satisfied. Therefore,~\eqref{eq:PEP_interpolation_constraints_3} can equivalently be written as 
\begin{equation}\label{eq:PEP_interpolation_constraints_4}
    \begin{aligned}
        & {\hbox{{\bf for each } $l\in\llbracket 1,m\rrbracket$}}\\ 
        &  \indentconstr\ba_{l}^{\top}(\bFcn-\bFcn_{+}) + \quadform{\bM_l}{(\XId{(C(I-A))}\Delta\bx  + \XId{(D-CB)}\bu - \XId{D}\bu_{+} + \XId{(CB\SumToZeroMat)}\hat{\bu}_{\star},\bu,\bu_{+})}\leq 0,\\ 
         &  \indentconstr\ba_{l}^{\top}(\bFcn_{+}-\bFcn) + \quadform{\bM_l}{(\XId{(C(A-I))}\Delta\bx  + \XId{(CB-D)}\bu + \XId{D}\bu_{+}-\XId{(CB\SumToZeroMat)}\hat{\bu}_{\star},\bu_{+},\bu)}\leq 0,\\ 
         &  \indentconstr\ba_{l}^{\top}(\bFcn-\bFcn_{\star}) + \quadform{\bM_l}{(\XId{C} \Delta\bx + \XId{D}\bu- \XId{(D\SumToZeroMat)}\hat{\bu}_{\star},\bu,\XId{\SumToZeroMat}\hat{\bu}_{\star})}\leq 0, \\ 
         &  \indentconstr\ba_{l}^{\top}(\bFcn_{\star}-\bFcn) + \quadform{\bM_l}{(-\XId{C} \Delta\bx -\XId{D}\bu + \XId{(D\SumToZeroMat)} \hat{\bu}_{\star},\XId{\SumToZeroMat}\hat{\bu}_{\star},\bu)}\leq 0,\\ 
         &  \indentconstr\ba_{l}^{\top}(\bFcn_{+}-\bFcn_{\star}) + \quadform{\bM_l}{ \XId{(CA)} \Delta\bx + \XId{(CB)}\bu + \XId{D} \bu_{+} - \XId{((D+CB)\SumToZeroMat)}\hat{\bu}_{\star},\bu_{+},\XId{\SumToZeroMat}\hat{\bu}_{\star})}\leq 0, \\ 
         &  \indentconstr\ba_{l}^{\top}(\bFcn_{\star}-\bFcn_{+}) + \quadform{\bM_l}{(-\XId{(CA)}\Delta\bx - \XId{(CB)}\bu - \XId{D} \bu_{+} + \XId{((D+CB)\SumToZeroMat)}\hat{\bu}_{\star},\XId{\SumToZeroMat}\hat{\bu}_{\star},\bu_{+})}\leq 0,\\ 
         & {\hbox{{\bf end}}}
    \end{aligned}
\end{equation}
and~\eqref{eq:worst_case_objective_quadratic_inserted_translated} can equivalently be written as 
\begin{align} \notag
        \Phi(\bxi,\bxi_{+},\bxi_{\star})  =& \quadform{Q_{\text{\o}}}{(\Delta\bx,\bu,\bu_{\star})}+q_{\text{\o}}^{\top}(\bFcn-\bFcn_{\star}) \\ \label{eq:worst_case_objective_quadratic_inserted_translated_SumToZeroMat}
        &+ \quadform{Q_{+}}{(\XId{A}\Delta\bx + \XId{B}\bu - \p{B\SumToZeroMat}_{\Id}\hat{\bu}_{\star},\bu_{+},\XId{\SumToZeroMat}\hat{\bu}_{\star})}+q_{+}^{\top}(\bFcn_{+}-\bFcn_{\star}).
\end{align}
If we let

\begin{align*}
    \bzeta&=(\Delta\bx,\bu,\bu_{+},\hat{\bu}_\star)\in\calH^{n}\times\calH^{m}\times\calH^{m}\times\calH^{m-1},\\
    \bm{\chi}&=(\bFcn-\bFcn_{\star},\bFcn_{+}-\bFcn_{\star})\in\reals^{m}\times\reals^{m},
\end{align*}
and use $\Sigma_{\text{\o}}$ and $\Sigma_{+}$ defined in~\eqref{eq:Sigma_matrices},~\eqref{eq:worst_case_objective_quadratic_inserted_translated_SumToZeroMat} can equivalently be written as 
\begin{align}\notag
\Phi(\bxi,\bxi_{+},\bxi_{\star})
        &=\quadform{Q_{\text{\o}}}{\p{\Sigma_{\text{\o}}}_{\Id}\bzeta}+\quadform{Q_{+}}{\p{\Sigma_{+}}_{\Id}\bzeta}+\bq^{\top}\bchi \\\notag
        &=\quadform{\Sigma_{\text{\o}}^{\top}Q_{\text{\o}}\Sigma_{\text{\o}}+\Sigma_{+}^{\top}Q_{+}\Sigma_{+}}{\bzeta}+\bq^{\top}\bchi \\ \label{eq:PEP_zeta_chi_cost}
        &=\quadform{\bQ}{\bzeta}+\bq^{\top}\bchi,
\end{align}
where $\bQ$ and $\bq$ are defined in~\eqref{eq:bQ_and_bq}. Using $E_{i,j}$ and $H_{i,j}$ defined in~\eqref{eq:E_matrices} and~\eqref{eq:H_matrices}, respectively,~\eqref{eq:PEP_interpolation_constraints_4} can equivalently be written as
\begin{align*}
    &{\hbox{\bf for each }} l\in\llbracket1,m\rrbracket{\hbox{ and distinct }} i, j\in\mathcal{I}\\
    &\indentconstr (H_{i,j}^{\top}\ba_l)^{\top}\bm{\chi} + \quadform{\bM_l}{\p{E_{i,j}}_{\Id}\bzeta}\leq 0, \\
    &{\hbox{\bf end}}
\end{align*}
which with $\bMlij=E_{i,j}^{\top}\bM_lE_{i,j}$ and $\balij=H_{i,j}^{\top}\ba_l$ (also defined in~\eqref{eq:Mlij_alij}) is equivalent to
\begin{equation}\label{eq:PEP_Mlij_constraint_form}
    \begin{aligned}
    &{\hbox{\bf for each }} l\in\llbracket1,m\rrbracket{\hbox{ and distinct }} i, j\in\mathcal{I}\\ 
    &\indentconstr \balij^{\top}\bm{\chi} + \quadform{\bMlij}{\bzeta}\leq 0. \\ 
    &{\hbox{\bf end}}
    \end{aligned}
\end{equation}

The equivalent reformulations~\eqref{eq:PEP_zeta_chi_cost} and~\eqref{eq:PEP_Mlij_constraint_form} give that~\eqref{eq:primalPEP} can be written as
\begin{equation}\label{eq:pre_pre_worst_case_SDP}
    \begin{aligned}
    & \underset{}{\text{maximize}}
    & & \quadform{\bQ}{\bzeta}+\bq^{\top}\bchi \\ 
    & \text{subject to} 
    & &{\hbox{\bf for each }} l\in\llbracket1,m\rrbracket{\hbox{ and distinct }} i, j\in\mathcal{I}\\ 
    &&&\indentconstr \balij^{\top}\bm{\chi} + \quadform{\bMlij}{\bzeta}\leq 0, \\
    &&&{\hbox{\bf end}} \\
    &&& \bzeta\in\calH^{n+3m-1},\,\bchi\in\reals^{2m}.
    \end{aligned}
\end{equation}

We define the \emph{Gramian function}
$g:\calH^k\to\sym^k_{+}$ such that $[g(\bz)]_{i,j}=\inner{z^{(i)}}{z^{(j)}}$ for each $i,j\in\llbracket1,k\rrbracket$ and $\bz=(z^{(1)},\ldots,z^{(k)})\in\calH^k$. If $M\in\sym^{k}$ and $\bz\in\calH^k$, then $\quadform{M}{\bz}=\trace\p{Mg(\bz)}$. Using this identity,~\eqref{eq:pre_pre_worst_case_SDP} can be written as
\begin{equation}\label{eq:pre_worst_case_SDP}
    \begin{aligned}
    & \underset{}{\text{maximize}}
    & & \trace\p{\bQ g(\bzeta)}+\bq^{\top}\bchi \\ 
    & \text{subject to} 
    && {\hbox{\bf for each }} l\in\llbracket1,m\rrbracket{\hbox{ and distinct }} i, j\in\mathcal{I}\\ 
    &&&\indentconstr \balij^{\top}\bm{\chi} + \trace\p{\bMlij g(\bzeta)}\leq 0, \\
    &&&{\hbox{\bf end}} \\
    &&& \bzeta\in\calH^{n+3m-1},\,\bchi\in\reals^{2m},
    \end{aligned}
\end{equation}
with optimal value equal to $S_{\Phi}^\star$. 
The problem
\begin{equation}\label{eq:worst_case_SDP}
    \begin{aligned}
    & \underset{}{\text{maximize}}
    & & \trace\p{\bQ G}+\bq^{\top}\bchi \\
    & \text{subject to} 
    && {\hbox{\bf for each }} l\in\llbracket 1,m\rrbracket {\hbox{ and distinct }} i, j\in\mathcal{I}\\ 
    & & &  \indentconstr\balij^{\top}\bm{\chi} + \trace\p{\bMlij G}\leq 0,\\
    & & &{\hbox{\bf end }} \\
    &&& G\in\sym^{n+3m-1}_{+},\,\bchi\in\reals^{2m},
    \end{aligned}
\end{equation}
is a relaxation of~\eqref{eq:pre_worst_case_SDP}, and therefore, has optimal value greater or equal to $S_{\Phi}^\star$. 

We will make use of the following fact: If $\dim\,\calH\geq k$, then $G\in\sym_+^k$ if and only if there exists $\bz\in\calH^k$ such that $G=g(\bz)$.
\cite[Lemma 3.1]{Ryu_2020} shows the result for the case $k=4$ and is based on the Cholesky decomposition of positive semidefinite matrices. The general case is a straightforward extension. This fact implies that if $\dim(\calH)\geq n+3m-1$, then~\eqref{eq:worst_case_SDP} has optimal value equal to $S_{\Phi}^\star$. Note that~\eqref{eq:worst_case_SDP} is a semidefinite program.

\paragraph{Dual problem and strong duality.} First, we derive the dual problem of~\eqref{eq:worst_case_SDP}. If we introduce dual variables $\lambda_{(l,i,j)}\geq 0$ for each $l\in\llbracket 1,m\rrbracket$ and distinct $i,j\in\mathcal{I}$ for the inequality constraints, the objective function of the dual problem becomes
\begin{align*}
    &\underset{G\in\sym^{n+3m-1}_{+},\,\bchi\in\reals^{2m}}{\sup}\p{
 \trace\p{\bQ G}+\bq^{\top}\bchi - \sum_{l=1}^{m}\sum_{\substack{i,j\in\mathcal{I}\\i\neq j}}\lambda_{(l,i,j)}\p{\balij^{\top}\bm{\chi} + \trace\p{\bMlij G}}}\\
    &=\underset{G\in\sym^{n+3m-1}_{+}}{\sup}\;\trace\p{\p{\bQ-\sum_{l=1}^{m}\sum_{\substack{i, j\in\mathcal{I}\\i\neq j}}\lambda_{(l,i,j)}\bMlij}G}+\underset{\bchi\in\reals^{2m}}{\text{sup}}\p{\bq -\sum_{l=1}^{m}\sum_{\substack{i, j\in\mathcal{I}\\i\neq j}}\lambda_{(l,i,j)}\balij}^{\top}\bchi.
\end{align*}
Since the dual problem is a minimization problem over the dual variables $\lambda_{(l,i,j)}$, we conclude that it can be written as
\begin{equation}\label{eq:worst_case_dual}
    \begin{aligned}
        & \underset{}{\text{minimize}}
        & & 0 \\
        & \text{subject to} && \lambda_{(l,i,j)}\geq 0 \text{ for each } l\in\llbracket 1,m\rrbracket \text{ and distinct }i,j\in\mathcal{I},\\
        & & & - \bQ + \sum_{l=1}^{m}\sum_{\substack{i, j\in\mathcal{I}\\i\neq j}}\lambda_{(l,i,j)}\bMlij\succeq 0,\\
        & & & - \bq + \sum_{l=1}^{m}\sum_{\substack{i, j\in\mathcal{I}\\i\neq j}}\lambda_{(l,i,j)}\balij=0,
    \end{aligned}
\end{equation}
which is a feasibility problem.

Next, suppose that the primal problem~\eqref{eq:worst_case_SDP} has a Slater point, i.e., there exists $G\in\sym_{++}^{n+3m-1}$ and $\bchi\in\reals^{2m}$ such that 
\begin{equation}\label{eq:slater_to_worst_case_SDP}
\begin{aligned}
    &\balij^{\top}\bm{\chi} + \trace\p{\bMlij G}\leq 0 \text{ for each } l\in\llbracket 1,m\rrbracket \text{ and distinct }i,j\in\mathcal{I}.
    \end{aligned}
\end{equation}
Then there is no duality gap, i.e., strong duality holds, between the primal problem~\eqref{eq:worst_case_SDP} and the dual problem~\eqref{eq:worst_case_dual}.

\paragraph{Alternatives.} The last step of the proof compares the optimal values of~\eqref{eq:primalPEP} and the dual problem~\eqref{eq:worst_case_dual}. We have established that $S_{\Phi}^\star$ is less than or equal to the optimal value of~\eqref{eq:worst_case_dual}. Thus, a sufficient condition for $S_{\Phi}\leq 0$ is that the dual problem~\eqref{eq:worst_case_dual} is feasible. In addition, if $\dim(\calH)\geq n+3m-1$ and there exists $G\in\sym_{++}^{n+3m-1}$ and $\bchi\in\reals^{2m}$ such that ~\eqref{eq:slater_to_worst_case_SDP} holds, the above condition also becomes a necessary condition.

This concludes the proof.
\end{proof}  

\section{Conclusions}\label{sec:conclusions}
We developed a flexible methodology for automated convergence analysis of a large class of first-order methods for solving convex optimization problems. The main result is a necessary and sufficient condition for the existence of a quadratic Lyapunov inequality \Update{within a predefined class of Lyapunov inequalities, which} amounts to solving a small-sized semidefinite program. The applicability and efficacy of the methodology are demonstrated by providing several new convergence results in Section~\ref{sec:numerical_examples}. 

We mention a few possible modifications that can be made to extend or modify the applicability and possibly improve the convergence results of the methodology. These were not pursued in the current work in order to maintain accessibility and not introduce unnecessary burdensome notation, but do constitute proper avenues for future works. First, each functional component $f_i$ in~\eqref{eq:the_problem_inclusion} can be modified to be from any function class that has quadratic interpolation constraints, e.g., the class of smooth functions~\cite{Taylor_Hendrickx_Glineur_2017b}, the class of convex and quadratically upper bounded functions~\cite{Baptiste_Taylor_Dieuleveut_2022}, the class of convex and Lipschitz continuous functions~\cite{Taylor_Hendrickx_Glineur_2017b}, etc. Second, the algorithm representation~\eqref{eq:linear_system_with_nonlinearity} can be extended to allow for more types of oracles (including, e.g., Frank--Wolfe-type oracles~\cite{Taylor_Hendrickx_Glineur_2017b}, Bregman-type oracles~\cite{dragomir2021optimal}, or approximate proximal point oracles~\cite{barre2022principled}) but also multiple evaluations of the same subdifferential $\partial f_{i}$ during the same iteration, enabling the analysis of, e.g., the forward--backward--forward splitting method of Tseng~\cite{Tseng_2000}. Third, similar to \cite{TaylorLessard_2018,Lessard_2016}, it is possible to extend the quadratic Lyapunov function and the quadratic residual function ansatzes to not only contain the current iterate $\bxi_{k}$, but some history $\bxi_{k}, \bxi_{k-1}, \ldots, \bxi_{k+1-h}$ for some integer $h\geq1$. This would allow exploring a greater class of Lyapunov inequalities that may lead to improved convergence results.

Finally, the methodology can be used in the process of finding analytical Lyapunov inequalities, convergence results\Update{, and optimal algorithm parameters}. \Update{ Indeed, finding a Lyapunov inequality is equivalent to solving a parametric semidefinite program. Obtaining a Lyapunov inequality involves discovering a closed-form solution for this semidefinite program, which can then be utilized to derive convergence results and select algorithm parameters. Works that aim to enable the obtaining of closed-form solutions include \cite{Henrion_Naldi_El_Din_LMI_2016,Henrion_Naldi_El_Din_SPECTRA_2019}, while a previous work focused on selecting algorithm parameters can be found in \cite{Vanscoy_2017}.}  

\section*{Acknowledgments}
This work was partially supported by the ELLIIT Strategic Research Area and the Wallenberg AI, Autonomous Systems and Software Program (WASP) funded by the Knut and Alice Wallenberg Foundation. S.~Banert and P.~Giselsson acknowledge support from Vetenskapsr\aa{}det grant VR~2021-05710. A.B.~Taylor acknowledges support from the French ``Agence Nationale de la Recherche'' as part of the ``Investissements d'avenir'' program, reference ANR-19-P3IA-0001 (PRAIRIE 3IA Institute), as well as support from the European Research Council (grant SEQUOIA 724063). We are grateful to Olle Kjellqvist for pointing out some of the details in Remark~\ref{rem:rank_n}. We also thank Anders Rantzer for providing comments and suggestions on an earlier version of this manuscript.  

\begingroup
\sloppy
\printbibliography
\endgroup

\end{document}